\documentclass[12pt,twoside,reqno]{amsart}
\usepackage[colorlinks=true, citecolor=blue, linkcolor=blue]{hyperref}
\usepackage{mathptmx, amsmath, amssymb, amsfonts, amsthm, mathptmx, enumerate, color,mathrsfs}
\usepackage{graphicx}
\usepackage{tikz}
\usepackage{lineno}
\usepackage{epstopdf}
\newcommand{\lap}{\mbox{$\Delta$}}
\newtheorem{theorem}{Theorem}[section]
\newtheorem{lemma}[theorem]{Lemma}

\newcommand{\be}{\begin{equation}}
\newcommand{\ee}{\end{equation}}

\theoremstyle{definition}

\newtheorem{remark}{Remark}
\numberwithin{equation}{section}
\allowdisplaybreaks[4]

\begin{document}
\setcounter{page}{1}

\vspace*{2.0cm}
\title[ Mixed local and nonlocal equations]
{On sliding methods for mixed local and nonlocal equations and Gibbons' conjecture}
\author[Yinbin Deng, Pengyan Wang, Zhihao Wang, Leyun Wu]{Yinbin Deng$^{a,b}$, Pengyan Wang$^a$, Zhihao Wang$^a$, Leyun Wu$^c$}
\maketitle
\vspace*{-0.6cm}

\begin{center}
{\footnotesize

$^a$  School of Mathematics and Statistics, Xinyang Normal University, Xinyang,   China\\

$^b$  School of Mathematics and Statistics, Central China Normal University, China \\

$^c$  School of Mathematics,  South China University of Technology,  Guangzhou,  China

}\end{center}

\vskip 4mm {\footnotesize \noindent {\bf Abstract.} We investigate  elliptic and parabolic equations involving mixed local and nonlocal operators of the form $(-\Delta)^s-\Delta$, as well as their parabolic counterparts with both the Marchaud fractional time derivative and the classical first-order derivative.

A major difficulty in this setting stems from the coexistence of operators with different nonlocal structures and incompatible scaling properties, which obstruct the direct use of classical sliding methods. To address this issue, we develop a refined sliding method suited to mixed local-nonlocal operators. As key technical ingredients, we establish new generalized weighted average inequalities, narrow region principles, and maximum principles in bounded and unbounded domains.

These tools enable us to derive monotonicity and one-dimensional symmetry results for mixed elliptic equations in bounded domains, half-spaces, and the whole space, and to extend the analysis to parabolic equations with mixed time derivatives. As an application, we resolve the Gibbons' conjecture for a class of mixed fractional equations.

\noindent {\bf Keywords.}
Monotonicity, one-dimensional symmetry, sliding methods, narrow region principle, maximum principle.

\noindent {\bf 2020 Mathematics Subject Classification.}
Primary 35R11, Secondary 35B50.}

\renewcommand{\thefootnote}{}
\footnotetext{
E-mail addresses: ybdeng@ccnu.edu.cn (Y. Deng); wangpy@xynu.edu.cn (P. Wang); wangzh082@126.com (Z. Wang); leyunwu@scut.edu.cn (L. Wu).}

\section{Introduction}

The mixed operator $(-\Delta)^s-\Delta $
arises naturally as the superposition of two fundamentally different diffusion mechanisms: the classical Laplacian, which models local Brownian motion, and the fractional Laplacian, which captures long-range jump processes.
In applications such as population dynamics and cell migration, the local diffusion term describes regular stochastic motion, whereas the nonlocal component accounts for abrupt long-distance dispersal or discontinuous diffusion. In financial models, including asset option pricing, the classical Brownian motion corresponds to the standard geometric Brownian motion, while the nonlocal jump term characterizes sudden price variations and extreme events. More broadly, across disciplines such as biology, finance, and materials science, mixed local–nonlocal operators provide a unified framework for modeling systems in which local smoothing effects and long-range interactions coexist. For further applications of mixed local and nonlocal operators, we refer the reader to \cite{BA25,BG2,BG1,CVW,CX1,DS1} and the references therein.

In this paper, we first investigate the monotonicity and one-dimensional symmetry of solutions to the following elliptic equation involving mixed local and nonlocal operators:
\begin{align*}
(-\Delta)^s u(x)-\Delta u(x) = g(x,u(x)),~x\in \Omega,
\end{align*}
where $\Omega\subset\mathbb{R}^n$ is either a bounded domain, the half-space $\mathbb{R}^n_+$, or the whole space $\mathbb{R}^n$. Here $-\Delta$ denotes the classical Laplacian, and the fractional Laplacian $(-\Delta)^s$ with $s\in(0,1)$ is defined by
\begin{align}\label{z1}
(-\Delta)^s  u(x ) &=C_{n,s} P.V. \int_{\mathbb R^n} \frac{ u(x )-u(y ) }{|x-y|^{n+2s}}dy \nonumber\\
&= \lim_{\epsilon \to 0} \int_{\mathbb{R}^n \setminus B_\epsilon(x)} \frac{u(x) - u(y)}{|x - y|^{n+2s}} dy,\end{align}
where $P.V.$ stands for the Cauchy principal value.
To ensure that the integral in \eqref{z1} is well defined, we assume
$$u \in C^2(\mathbb R^n) \cap \mathcal{L}_{2s},$$
where
$$ \mathcal{L} _{2s}=\{u  \in L^{1}_{loc} (\mathbb{R}^n) \mid \int_{\mathbb R^n} \frac{|u(x )|}{1+|x|^{n+2s}}dx<+\infty\}.$$

Another objective of this article is to study monotonicity and one-dimensional symmetry  of solutions to the following parabolic equation involving mixed local and nonlocal operators:
\begin{align}\label{eq1}
\partial^{\alpha}_t u (x, t)+(-\Delta)^s u(x,t)-\Delta u(x,t)=h(t,u(x,t)),~(x,t)\in \Omega \times \mathbb{R},
\end{align}
where $0<s<1$, $0<\alpha\le1$, and $\Omega$ may be either bounded or unbounded. The operator $\partial_t^{\alpha}$ denotes the Marchaud fractional derivative of order $\alpha$, defined by
$$ {\partial^\alpha_ t}u(x,t)=C_\alpha \int_{-\infty}^t \frac{u(x,t)-u(x,\tau)}{(t-\tau)^{1+\alpha}}d\tau,$$
with $\alpha\in(0,1)$ and $C_{\alpha}=\frac{\alpha}{\Gamma(1-\alpha)}$, where $\Gamma$ denotes the Gamma function. When $\alpha=1$, the Marchaud fractional derivative reduces to the classical first-order time derivative $\partial_t u$.

In the parabolic setting, $(-\Delta)^s u(x, t)$  is defined in the same way as in the elliptic case. For each fixed $t\in \mathbb R$,
\begin{align*}
(-\lap)^s u(x,t)&=C_{n,s} P.V. \int_{\mathbb R^n} \frac{u(x,t)-u(y,t)}{|x-y|^{n+2s}}dy\\
&= C_{n,s}  \lim\limits_{\epsilon\rightarrow 0}\int_{\mathbb R^n\setminus B_{\epsilon}(x)}\frac{u(x,t)-u(y,t)}{|x-y|^{n+2s}}dy.
\end{align*}
To guarantee that the left-hand side of \eqref{eq1} is well defined, we assume
$$
u(x,t)\in (C^2(\mathbb R^n)\cap {\tilde {\mathcal L}}_{2s})\times (C^1(\mathbb R)\cap {\mathcal L}^-_\alpha (\mathbb R)),
$$
where
$$ {\tilde {\mathcal L}}_{2s}=\left\{u(\cdot, t) \in L^1_{loc} (\mathbb{R}^n) \mid \int_{\mathbb R^n} \frac{|u(x,t)|}{1+|x|^{n+2s}}dx<+\infty\right\}$$
and
$${\mathcal L}^-_{\alpha}(\mathbb R):=\left\{ u\in L^1_{loc}(\mathbb R ) \mid \int_{-\infty}^t \frac{|u(\tau)|}{1+|\tau|^{1+\alpha}}d\tau<+\infty~\mbox{ for each }~t\in \mathbb R\right\}.$$

Symmetry and monotonicity are fundamental qualitative properties of solutions to partial differential equations. To overcome the intrinsic nonlocality of the fractional Laplacian, a variety of powerful techniques have been developed in recent years to study symmetry and monotonicity of solutions to fractional equations. These include the extension method \cite{CS}, the method of moving planes in integral form \cite{CLO,LZ1}, the direct method of moving planes \cite{CL,CLL}, sliding methods \cite{LiuC,WC2,W21}, the asymptotic method of moving planes \cite{CWNH,WpC}, the method of moving spheres \cite{CLZr}, the method of scaling spheres \cite{DQ}, and other related approaches \cite{LWX,LZ2}.

Among these techniques, the sliding method is a particularly effective tool for establishing monotonicity and one-dimensional symmetry of solutions to partial differential equations. This method originates from the pioneering work of Berestycki and Nirenberg \cite{BN1,BN3}. Subsequently, Wu and Chen \cite{WC2} and Liu \cite{LiuC} further developed this approach to derive monotonicity and one-dimensional symmetry of solutions to the elliptic fractional
$p$-equation
\begin{equation}\label{101}
(-\Delta)_p^s u(x)=f(x,u(x))
\end{equation}
in unbounded domains.

In the case $p=2$, Wu and Chen \cite{WC3} established corresponding results for equation \eqref{101} in both bounded and unbounded domains.
When $p=2$ and $f=f(x,u(x),\nabla u(x))$, Wang \cite{W21} investigated the monotonicity and uniqueness of solutions to \eqref{101} in bounded domains and in the half-space $\mathbb{R}^n_+$ by means of the sliding method.
For $2<p<+\infty$ and nonlinearities of the form $f=f(x,u(x),\nabla u(x))$, Wang \cite{W23} provided a new proof of monotonicity and uniqueness for solutions to elliptic fractional $p$-equations \eqref{101} via the sliding method.

More recently, Chen and Wu \cite{CWA} presented several applications of the sliding method. In particular, they established monotonicity and one-dimensional symmetry of entire solutions to the fractional diffusion equation
$$\partial_t u(x,t)+(-\Delta)^s u(x,t)=f(t,u(x,t)).$$
Subsequently, Guo \cite{G24} proved the one-dimensional symmetry and monotonicity of entire positive solutions to parabolic fractional $p$-equations by applying the sliding method.
For a comprehensive treatment of the theory and applications of the sliding method, we refer the reader to \cite{CHM,DS,FM,GMZ}.

Next, we present our main results, distinguishing between the elliptic and parabolic equations involving mixed local and nonlocal operators.

\subsection{  Elliptic equations involving  mixed local and nonlocal operators}
\
\newline
\indent
The narrow region principle is a fundamental ingredient of the sliding method, as it provides a starting position from which the domain can be slid. Therefore, we first establish a narrow region principle for elliptic equations involving mixed local and nonlocal operators.

\begin{theorem}\label{D1} (Narrow region principle in bounded domains)
Let $\Omega$ be a bounded narrow region in $\mathbb{R}^n$. Assume that $w_\tau \in \mathcal{L}_{2s} \cap C^2(\Omega)$  is upper semi-continuous on $\overline{\Omega}$ and satisfies
\begin{equation}\label{D11}
\left\{
\begin{array}{ll}
(-\Delta)^s w_\tau (x)-\Delta w_\tau (x) +c(x)w_\tau (x)\leq 0, ~&x\in \Omega,\\
w_\tau(x) \leq 0,~&x\in\Omega^c:=\mathbb{R}^n \setminus \Omega,\\
 \end{array}
\right.
\end{equation}
where $c(x)$ bounded from below in $\Omega$. Let $l_n(\Omega)$
denote the width of $\Omega$ in the $x_n$-direction, and assume that $\Omega$ is narrow in this direction in the sense that
\begin{align}\label{D12}
l_n(\Omega)|\inf_\Omega c(x)|^{\frac{1}{2s}}\leq C,
\end{align}
where $C$ is a positive constant.

Then
\begin{align}\label{D13}
w_\tau(x) \leq 0,~x\in \Omega,
\end{align}
furthermore, we have
\begin{align}\label{D14}
\mbox{ either }~ w_\tau(x) < 0,~x\in \Omega~\mbox{ or }~ w_\tau(x) \equiv 0,~x\in \mathbb{R}^n.
\end{align}
\end{theorem}

\begin{remark}
Theorem~\ref{D1} can be regarded as an extension of Theorem~1.1 in \cite{WC3}.
\end{remark}

In addition, following \cite{BN3, WC2,WC3}, we impose an exterior condition on the solution $u$. Let
$$u(x)= \eta(x),~x\in \Omega^c,$$
and assume that

$(\textbf{A})$
For any three points $x=(x',x_n)$, $y=(x',y_n)$, and $z=(x',z_n)$ lying on a line segment parallel to the $x_n$-axis, with $y_n<x_n<z_n$ and $y,~z\in \Omega^c$, we have
\begin{align*}
\eta(y) < u(x) < \eta(z), ~x \in \Omega
\end{align*}
and
\begin{align*}
\eta(y) \leq \eta(x) \leq \eta(z), ~x \in \Omega^c.
\end{align*}

Applying the sliding methods, we obtain the monotonicity of solutions to  elliptic equations involving mixed  local and nonlocal operators in  bounded domains.

\begin{theorem}\label{D2} (Monotonicity in bounded domains)
Let $\Omega$ be a bounded domain in $\mathbb{R}^n$ that is convex in the $x_n$-direction. Assume that $u\in \mathcal{L}_{2s}\cap C^2(\Omega)$ is a solution of
\begin{equation}\label{D21}
\left\{
\begin{array}{ll}
(-\Delta)^s u(x)-\Delta u(x)=g(u(x)), ~&x\in \Omega,\\
u(x)=\eta(x),~&x\in \Omega^c,\\
 \end{array}
\right.
\end{equation}
and satisfies $(\textbf{A})$. Suppose that $g$ is Lipschitz continuous. Then
\begin{align*}
u~\mbox{ is strictly monotone increasing with respect to }~x_n~\mbox{ in }~\Omega,
\end{align*}
that is
\begin{align*}
u(x',x_n+\tau)>u(x',x_n),~(x',x_n),(x',x_n+\tau)\in \Omega~\mbox{ for any }~\tau > 0.
\end{align*}
\end{theorem}

To apply the sliding method in unbounded domains, we establish the following maximum principle.

\begin{theorem} \label{D3} (Maximum principle in unbounded domains)
Let $\Omega \subset \mathbb{R}^n$ be an open set, possibly unbounded and disconnected. For any $\tilde x \in \mathbb{R}^n$, assume that
\begin{align}\label{D31}
\underset{R \to +\infty}{\underline{\lim}} \frac{| B_{R}(\tilde x)\cap \Omega^c|}{|B_{R}(\tilde x)|}\geq c_0> 0.
\end{align}

Let $u\in \mathcal{L}_{2s}\cap C^2(\Omega)$ be bounded from above and satisfy
\begin{equation}\label{D32}
\left\{
\begin{array}{ll}
(-\Delta)^s u(x)-\Delta u(x) +c(x)u(x)\leq 0, &~\mbox{ at the points in } \Omega~\mbox{ where }~u(x) > 0,\\
u(x) \leq 0 ,&~x\in \Omega^c,\\
 \end{array}
\right.
\end{equation}
for some nonnegative function $c(x),$
then
\begin{align}\label{D33}
u(x) \leq 0,~x\in \Omega.
\end{align}
\end{theorem}

Using the above maximum principle, we obtain monotonicity results in unbounded domains.

\begin{theorem}\label{D4} (Monotonicity in the upper-half space)
Let $u(x) \in \mathcal{L}_{2s} \cap C^2(\mathbb{R}_+^n) \cap C(\overline{\mathbb{R}_+^n})$ be a nonnegative bounded solution of
\begin{equation}\label{D41}
\left\{
\begin{array}{ll}
(-\Delta)^s u(x)-\Delta u(x) = g(u(x)), ~& x \in \mathbb{R}_+^n :=\{x=(x',x_n)\in \mathbb R^n \mid x_n>0\}, \\
u(x) > 0, ~& x\in \mathbb{R}_+^n , \\
u(x) = 0, ~& x \in \mathbb{R}^n \setminus \mathbb{R}_+^n,
 \end{array}
\right.
\end{equation}
where $g$ is continuous and monotone decreasing. Then
\begin{align*}
u(x) ~\mbox{ is strictly increasing with respect to }~ x_n~\mbox { in}~ \mathbb{R}_+^n.
\end{align*}
\end{theorem}

\begin{theorem}\label{D5}
Let $u\in \mathcal{L}_{2s}\cap C^2(\mathbb{R}^n)$ be a uniformly continuous solution of
\begin{align}\label{D51}
(-\Delta)^s u(x)-\Delta u(x)=g(x,  u(x)), ~x\in \mathbb{R}^n
\end{align}
satisfying
\begin{align}\label{D52}
|u(x)| \leq 1~\mbox{ and }~u(x', x_n) \underset{x_n \to \pm \infty}{\longrightarrow} \pm 1~\mbox{ uniformly in }~ x' = (x_1, \cdots, x_{n-1}).
\end{align}

Suppose that $g$ is continuous and that there exists $\gamma>0$ such that
\begin{align}\label{D53}
g~\mbox{ is nonincreasing on }~u\in  [-1, -1 + \gamma]~\mbox{ and on }~u\in [1 - \gamma, 1]
\end{align}
and
\begin{align}\label{D54}
g(x',x_n,u)\leq g(x',\hat x_n,u),~x=(x',x_n)\in \mathbb R^n,~x_n\leq \hat x_n,
\end{align}
Then $u$ is strictly increasing with respect to $x_n$ and depends only on $x_n$.
\end{theorem}

\begin{remark}
Theorem~\ref{D5} extends Theorem~1.2 in \cite{BVDV}, where stronger regularity and structural assumptions on $g$ are imposed. In particular, \cite{BVDV} assumes that $g(x,u(x))=g(u),~u\in C^3(\mathbb R^n) \cap W^{4,\infty}(\mathbb R^n)$,  and  $g\in C^1(\mathbb R)$ with $ \underset{|u|\geq 1}{\sup}~g'(u)<0.$
\end{remark}

\begin{remark}
Theorem~\ref{D5} is closely related to the well-known \textbf{Gibbons' conjecture} \cite{GG}. This conjecture asserts that if $u$ solves
\begin{align}\label{G1}
-\Delta u(x)=u(x)-u^3(x),~x\in \mathbb{R}^n
\end{align}
 satisfies $|u(x)| \leq 1$, and
\begin{align}\label{G2}
u(x', x_n) \underset{x_n \to \pm \infty}{\longrightarrow} \pm 1~\mbox{ uniformly in }~ x' \in \mathbb{R}^{n-1},
\end{align}
then $u$ is monotone in the $x_n$-direction and depends only on $x_n$. This conjecture can be regarded as a weakened form of \it{De Giorgi’s conjecture}. In particular, when $g(x,u)=u-u^{3}$, equation \eqref{D51}  reduces to a mixed fractional Allen--Cahn equation.

It is worth noting that the classical \it{De Giorgi's conjecture} corresponds to a stronger formulation of this problem. More precisely, it is referred to as De Giorgi’s conjecture if condition \eqref{G2} is replaced by
\begin{align}\label{G3}
\lim_{x_n \to \pm \infty} u(x',x_n)=\pm 1,\quad \forall\, x'\in \mathbb{R}^{n-1},
\quad \frac{\partial u}{\partial x_n}>0.
\end{align}
As an illustrative example, choosing $g(x,u)=u-u^{3}$ in \eqref{D51} yields the mixed fractional Allen-Cahn equation.
\end{remark}

\subsection{  Parabolic equations involving mixed local and nonlocal operators}
\
\newline
\indent In this section, we first establish a narrow region principle for parabolic equations involving mixed local and nonlocal operators. This principle plays a crucial role in the application of the sliding method in the space-time domain $\Omega \times \mathbb{R}$.

\begin{theorem}\label{N1} (Narrow region principle in unbounded domains)
Let $\Omega\subset \mathbb{R}^n$ be a bounded  domain. Suppose that $U_\tau (x,t)=u_\tau(x,t)-u(x,t)\in (C^2(\Omega)\cap {\tilde {\mathcal L}}_{2s})\times (C^1(\mathbb R)\cap {\mathcal L}^-_\alpha (\mathbb R))$  is bounded from below in $\overline{\Omega}\times \mathbb R$ and for any $t\in \mathbb R$, $U_\tau(\cdot,t)$ is lower semicontinuous on $\overline{\Omega}$.
Assume that
\begin{equation}\label{N11}
\left\{
\begin{array}{ll}
{\partial^\alpha_ t}U_\tau(x, t)+(-\Delta)^s U_\tau (x,t)-\Delta U_\tau (x,t) +C(x,t)U_\tau (x,t)\geq0, ~&(x,t)\in\Omega \times \mathbb{R},\\
U_\tau (x,t)\geq 0,~&(x,t)\in\Omega^c \times \mathbb{R},\\
 \end{array}
\right.
\end{equation}
where $C(x,t)$ is bounded from below. Suppose further that $\Omega$ is narrow in the $x_n$-direction and satisfies
\begin{equation}\label{1122}
    l_n(\Omega)|\underset{\Omega\times \mathbb R}{\inf}C(x,t)|^{\frac{1}{2s}}\leq C,
\end{equation}
where $l_n(\Omega)$ denotes the width of $\Omega$ in the $x_n$-direction and $C>0$ is a constant. Then
\begin{align}\label{N13}
U_\tau(x,t) \geq 0,~(x, t) \in \Omega \times \mathbb{R}.
\end{align}
\end{theorem}

As in the elliptic case, we impose an exterior condition on $u$. Let
$$u(x,t)=\theta(x,t),~x\in \Omega^c,~t \in\mathbb{R}.$$
We further assume the following condition:
$(\textbf{B})$ For any three points $x = (x', x_n)$, $y = (x', y_n)$ and $z = (x', z_n)$ lying on a line segment parallel to the $x_n$-axis, where $y_n < x_n < z_n$ for  $y,~ z \in \Omega^c$, one has
\begin{align*}
\theta(y,t) < u(x,t) <\theta(z,t), ~x \in \Omega,~t \in\mathbb{R}
\end{align*}
and
\begin{align*}
\theta(y,t) \leq \theta(x,t) \leq \theta(z,t), ~x \in \Omega^c,~t\in\mathbb{R}.
\end{align*}

Applying the sliding method, we obtain monotonicity results for parabolic equations involving mixed local and nonlocal operators in  unbounded domains $\Omega \times \mathbb R$.

\begin{theorem}\label{N2} (Monotonicity in bounded domains)
Let $\Omega\subset\mathbb{R}^n$ be a bounded domain that is convex in the $x_n$-direction. Assume that $0<s<1$, $0<\alpha\le 1$, and that
 $u(x,t)\in (C^2(\Omega)\cap {\tilde {\mathcal L}}_{2s})\times (C^1(\mathbb R)\cap {\mathcal L}^-_\alpha (\mathbb R))$ is a solution  of
\begin{equation}\label{N21}
\left\{
\begin{array}{ll}
{\partial^\alpha_ t}u(x, t)+(-\Delta)^s u(x)-\Delta u(x)=h(t,u(x,t)), ~&(x, t) \in \Omega \times \mathbb{R},\\
u(x,t)=\theta(x,t),~&(x,t)\in (\mathbb{R}^n \setminus \Omega)\times \mathbb{R},\\
 \end{array}
\right.
\end{equation}
and satisfies $(\textbf{B})$.
Suppose that $h$ is Lipschitz continuous. Then $u$ is strictly increasing with respect to $x_n$ in $\Omega\times\mathbb{R}$, that is,
\begin{align*}
u(x',x_n+\tau,t)>u(x',x_n,t),~(x',x_n,t),(x',x_n+\tau,t)\in \Omega\times \mathbb{R},~\mbox{ for any }~\tau > 0.
\end{align*}
\end{theorem}

To apply the sliding method in unbounded spatial domains, we establish the following maximum principle.

\begin{theorem} \label{T2} (Maximum principle in unbounded domains)
Let $D \subset \mathbb{R}^n$ be an open set, possibly unbounded and disconnected. For any $x^0 \in \mathbb{R}^n$,
\begin{align}\label{T21}
\underset{R \to +\infty}{\underline{\lim}}\frac{|B_R(x^0) \cap D^c|}{|B_R(x^0)|} \geq c_0 > 0.
\end{align}

Let  $0<s<1,~0<\alpha \leq 1,$ $u(x,t)\in (C^2(D)\cap {\tilde {\mathcal L}}_{2s})\times (C^1(\mathbb R)\cap {\mathcal L}^-_\alpha (\mathbb R))$ be bounded from above and satisfy
\begin{equation}\label{T22}
\left\{
\begin{array}{ll}
{\partial^\alpha_ t}u(x, t) + (-\Delta)^s u(x, t)-\Delta u(x,t) \leq 0, & \mbox{ at the points in } D\times \mathbb{R} ~\mbox{ where }~u(x,t) > 0,\\
u(x,t) \leq 0, &\mbox{ in } D^c \times \mathbb{R},
 \end{array}
\right.
\end{equation}
then
\begin{align}\label{T23}
u(x, t) \leq 0, ~(x,t)\in D \times \mathbb{R}.
\end{align}
\end{theorem}

As an immediate consequence, we obtain monotonicity in the upper half-space.
\begin{theorem}\label{s1} (Monotonicity in the upper-half space)
Let  $0<s<1,~0<\alpha \leq 1,$ $u(x,t)\in \big(C^2(\mathbb R_+^n)\cap {\tilde {\mathcal L}}_{2s}\cap C(\overline{\mathbb{R}_+^n})\big)\times (C^1(\mathbb R)\cap {\mathcal L}^-_\alpha (\mathbb R))$
 be a nonnegative bounded solution of
\begin{equation}\label{s11}
\left\{
\begin{array}{ll}
{\partial^\alpha_ t}u(x, t) + (-\Delta)^s u(x, t)-\Delta u(x,t) = h(t, u(x, t)), ~& (x, t) \in \mathbb{R}_+^n \times \mathbb{R}, \\
u(x, t) > 0, ~& (x, t) \in \mathbb{R}_+^n \times \mathbb{R}, \\
u(x, t) = 0, ~& (x, t) \in (\mathbb{R}^n \setminus \mathbb{R}_+^n) \times \mathbb{R},
 \end{array}
\right.
\end{equation}
where $h(t,u)$ is continuous and nonincreasing in $u$. Then
\begin{align*}
u(x,t) ~\mbox{ is strictly increasing with respect to }~ x_n~\mbox{in}~\mathbb{R}_+^n \times \mathbb{R}.
\end{align*}
\end{theorem}

Applying the above maximum principle, we obtain a parabolic analogue of Gibbons' conjecture.

\begin{theorem}\label{T3} (One-dimensional symmetry)
Let $0<s<1$, $0<\alpha\le 1$, and let
 $ u(x,t)\in (C^2(\mathbb R^n)\cap {\tilde {\mathcal L}}_{2s})\times (C^1(\mathbb R)\cap {\mathcal L}^-_\alpha (\mathbb R))$ be a uniformly continuous solution of
\begin{align}\label{T31}
\partial^\alpha_tu(x, t) + (-\Delta)^s u(x, t)-\Delta u(x,t) = h(t, u(x, t)), ~ (x, t) \in \mathbb{R}^n \times \mathbb{R}
\end{align}
satisfying $|u(x,t)|\le 1$ and
\begin{align}\label{T32}
u(x', x_n, t) \underset{x_n \to \pm\infty}{\longrightarrow} \pm 1 ~\mbox{ uniformly in }~ x' = (x_1, \cdots, x_{n-1}) ~\mbox{ and in } t.
\end{align}

Suppose  that $h(t, u)$ is continuous in $\mathbb{R} \times ([-1, 1])$ and that  for each fixed $t \in \mathbb{R} $,
\begin{align}\label{T33}
h(t, u) ~\mbox{ is nonincreasing for }~ |u| \geq 1 - \gamma ~\mbox{ with some }~ \gamma > 0.
\end{align}
Then $u$ is strictly increasing with respect to $x_n$ and depends only on $x_n$, that is, $u(x, t) = u(x_n, t).$
\end{theorem}
\begin{remark}
Theorem \ref{T3} can be seen as an extension of Theorem 1.3 in \cite{CWA}.
\end{remark}

Finally, we emphasize that rescaling techniques play a crucial role in the qualitative analysis of solutions. While the fractional Laplacian and the classical Laplacian scale differently, a key contribution of this work lies in adapting the sliding method to accommodate these distinct scaling behaviors within a unified framework. This approach allows us to resolve the Gibbons-type conjecture for mixed fractional equations and may be applicable to a broader class of mixed elliptic and parabolic problems.

The paper is organized as follows. In Section~2, we establish narrow region principles (Theorems~\ref{D1} and~\ref{N1}), maximum principles (Theorems~\ref{D3} and~\ref{T2}), and generalized average inequalities. In Section~3, we investigate monotonicity and one-dimensional symmetry for elliptic equations (Theorems~\ref{D2}, \ref{D4}, and~\ref{D5}). Section~4 is devoted to the parabolic case, where we prove monotonicity and symmetry results (Theorems~\ref{N2}, \ref{s1}, and~\ref{T3}).

\section{Various maximum principles and key lemmas}

In this section, we establish several narrow region principles (Theorems~\ref{D1} and~\ref{N1}), maximum principles (Theorems~\ref{D3} and~\ref{T2}), as well as some generalized average inequalities.

Before stating our main results, we introduce some notation. For $x = (x', x_n)$, where $x' = (x_1, \ldots, x_{n-1}) \in \mathbb{R}^{n-1}$, and for $\tau > 0$, we define
$$x_\tau := x + \tau e_n, \qquad \text{where } e_n = (0',1).$$

\subsection{Maximum principles and key lemmas  involving elliptic  mixed local and nonlocal operators}
\
\newline
\indent
Let
$$u_\tau(x) := u(x',x_n+\tau), \qquad w_\tau(x) := u(x)-u_\tau(x).$$
Since the classical generalized weighted average inequality for the fractional Laplacian is no longer applicable in the present setting, we first establish a generalized average inequality adapted to the mixed local and nonlocal operators.

\begin{lemma}\label{L2}
Assume that $u \in \mathcal{L}_{2s}\cap C^{2}(\mathbb{R}^n)$ attains its maximum at a point $\tilde{x}\in \mathbb{R}^n$. Then, for any $r>0$, the following inequality holds:
\begin{align}\label{L21}
\frac{C_0}{C_{n,s}}\, r^{2s}\big\{(-\Delta)^s u(\tilde{x})-\Delta u(\tilde{x})\big\}
+ C_0 r^{2s}\int_{B_r^c(\tilde{x})} \frac{u(y)}{|\tilde{x}-y|^{n+2s}}\,dy
\ge u(\tilde{x}),
\end{align}
where $B_r^c(\tilde{x})$ denotes the complement of $B_r(\tilde{x})$ in $\mathbb{R}^n$,
$$C_0 := \left(\int_{B_1^c(0)} \frac{1}{|y|^{n+2s}}\,dy\right)^{-1}$$
is a positive constant, and
$$C_0 r^{2s}\int_{B_r^c(\tilde{x})} \frac{1}{|\tilde{x}-y|^{n+2s}}\,dy = 1.$$
\end{lemma}

\begin{proof}
Suppose that $\tilde{x}$ is a maximum point of  $u(x)$, applying the definitions of the fractional Laplacian and Laplacian, we have
\begin{align*}
(-\Delta)^s u(\tilde{x})-\Delta u(\tilde{x})
\geq& (-\Delta)^s u(\tilde{x})\\
=& C_{n,s}P.V. \int_{\mathbb{R}^n} \frac{u(\tilde{x}) - u(y)}{|\tilde{x} - y|^{n+2s}}dy \\
\geq& C_{n,s} \int_{B_r^c(\tilde{x})} \frac{u(\tilde{x}) - u(y)}{|\tilde{x} - y|^{n+2s}}dy \\
=& C_{n,s} u(\tilde{x}) \int_{B_r^c(\tilde{x})} \frac{1}{|\tilde{x} - y|^{n+2s}}dy - C_{n,s} \int_{B_r^c(\tilde{x})} \frac{u(y)}{|\tilde{x} - y|^{n+2s}}dy \\
=& C_{n,s} \frac{u(\tilde{x})}{r^{2s}} \int_{B_1^c(0)} \frac{1}{|y|^{n+2s}}dy - C_{n,s} \int_{B_r^c(\tilde{x})} \frac{u(y)}{|\tilde{x} - y|^{n+2s}}dy \\
=& \frac{C_{n,s}}{C_0} \frac{u(\tilde{x})}{r^{2s}} - C_{n,s} \int_{B_r^c(\tilde{x})} \frac{u(y)}{|\tilde{x} - y|^{n+2s}}dy,
\end{align*}
where $$C_0 = \frac{1}{\int_{B_1^c(0)} \frac{1}{|y|^{n+2s}}dy}$$ is a positive constant. Thus,
\begin{align*}
\frac{C_0}{C_{n,s}} r^{2s} \{(-\Delta)^s u(\tilde{x})-\Delta u(\tilde{x})\} + C_0 r^{2s}\int_{B_r^c(\tilde{x})} \frac{u(y)}{{|\tilde{x} - y|^{n+2s}}}dy \geq u(\tilde{x}),
\end{align*}
which yields  \eqref{L21} and completes the proof of Lemma \ref{L2}.
\end{proof}

\begin{proof}[Proof of Theorem \ref{D1}]

 If \eqref{D13} does not hold, then by the lower semi-continuity of $w_\tau$ in $\Omega$,  there exists a point $x_0$ such that
\begin{align*}
w_\tau(x_0) := \max_{\overline{\Omega}} w_\tau (x)> 0,
\end{align*}
that is,
\begin{align}\label{D101}
-w_\tau(x_0) = \min_{\overline{\Omega}} (-w_\tau(x)) < 0.
\end{align}
Combining  \eqref{D11}, \eqref{D12}, \eqref{D101} and  $\lap w_{\tau}(x_0)\leq 0$, we have
\begin{align*}
&-\{(-\Delta)^s w_\tau(x_0)-\Delta w_\tau(x_0)+ c(x_0)w_\tau(x_0)\} \\
\leq&-(-\Delta)^s w_\tau(x_0)-c(x_0)w_\tau(x_0)\\
=&C_{n,s}P.V. \int_{\Omega} \frac{-w_\tau(x_0) +w_\tau(y)}{|x_0 - y|^{n+2s}} dy\\
&+C_{n,s}P.V. \int_{\mathbb{R}^n \setminus \Omega} \frac{-w_\tau(x_0) +w_\tau(y)}{|x_0 - y|^{n+2s}} dy-c(x_0)w_\tau(x_0)\\
\leq &-w_\tau(x_0)C_{n,s}P.V. \int_{\mathbb{R}^n \setminus \Omega} \frac{ 1}{|x_0 - y|^{n+2s}} dy-\inf_{\Omega} c(x)w_\tau(x_0)\\
\leq &-w_\tau(x_0)\{\frac{C}{l_n^{2s}(\Omega)}+\inf_{\Omega} c(x)\}\\
<&0,
\end{align*}
that is,
$$(-\Delta)^s w_\tau(x_0)-\Delta w_\tau(x_0)+ c(x_0)w_\tau(x_0)>0,$$
which contradicts the first inequality in \eqref{D11}. Hence, we infer that \eqref{D13} must hold, that is,
$$w_\tau(x) \le 0, \qquad x \in \Omega.$$

Based on the above result, if $w_\tau(x)=0$ at some point $\bar{x}\in\Omega$, then $\bar{x}$ is a maximum point of $w_\tau$ in $\Omega$. If $w_\tau \not\equiv 0$ in $\mathbb{R}^n$, then
$$(-\Delta)^s w_\tau(\bar{x})-\Delta w_\tau(\bar{x})+c(\bar{x})w_\tau(\bar{x})
\ge (-\Delta)^s w_\tau(\bar{x})
= C_{n,s}\,\mathrm{P.V.}\int_{\mathbb{R}^n}\frac{-w_\tau(y)}{|\bar{x}-y|^{n+2s}}\,dy
>0,$$
which contradicts the first inequality in \eqref{D11}. Therefore, \eqref{D14} holds.

This completes the proof of Theorem~\ref{D1}.
\end{proof}

\begin{proof}[Proof of Theorem \ref{D3}]
We prove \eqref{D33} by contradiction. Suppose that \eqref{D33} does not hold. Since $u(x)$ is bounded from above in $\Omega$, there exists a positive constant $M$ such that
\begin{align}\label{D301}
\sup_{x \in \Omega} u(x) := M > 0.
\end{align}
Then there exists a sequence $\{x^k\} \subset \Omega$ such that
$$u(x^k) \to M ~ \mbox{ as }~ k \to \infty.$$

Let
$$
\psi(x) :=
\begin{cases}
b \, e^{\frac{1}{|x|^2-1}}, ~& |x|<1,\\
0, ~& |x|\ge 1,
\end{cases}
$$
where $b>0$ is a constant. Note that
$$\psi(0) = \max_{\mathbb{R}^n} \psi(x) = 1 ~\mbox{ with }~ b = e.$$
Clearly, $\psi(x)$ is radially decreasing with respect to $|x|$ and has support in $B_1(0)$.

We then rescale $\psi(x)$ by defining
$$\psi_k(x) := \psi\left( \frac{x - x^k}{r_k} \right), ~ \mbox{ where }~ r_k = \frac{1}{2} \operatorname{dist}(x^k, \partial \Omega) > 0.$$

Now we assert that $r_k$ is bounded away from $0$, i.e., there exists a positive constant $\bar c$ such that
\begin{align}\label{D311}
    r_k=\frac{1}{2}dist(x^k, \partial \Omega)\geq \bar c >0,~k\in \mathbb{N}.
\end{align}
Otherwise,  $\{x^k\}$ will converge to a point on $\partial \Omega$. Then  by $u\leq 0,~x\in \Omega^c$, we have $$u(x^k) \to 0,~\mbox{ as }~k\to +\infty,$$
which contradicts with the fact that $$u(x^k) \to M>0, ~\mbox{ as }~ k \to \infty.$$ Thus, \eqref{D311} is valid.

Next,  we choose a nonnegative sequence $\{\varepsilon_k\} \searrow 0$ and introduce the  auxiliary function
\begin{align}\label{D303}
U_k(x) = u(x) + \varepsilon_k \psi_{r_k}(x),
\end{align}
where $r_k=\frac{1}{2}dist(x^k, \partial \Omega)\geq \bar c >0,~k\in \mathbb{N}.$

 Noting that
$$U_k(x) \leq M,~ x \in \mathbb{R}^n \setminus B_{r_k}(x^k)$$
and
$$U_k(x^k)=u(x^k)+\varepsilon_k \psi_k(x^k)>M\geq U_k(x).$$
It is then clear that $U_k(x)$ attains its maximum value in $B_{r_k}(x^k)$. Without loss of generality, let $\tilde{x}^k \in B_{r_k}(x^k)$ be a point such that
\begin{align}\label{D304}
M+ \varepsilon_k \geq U_k(\tilde{x}^k) := \underset{\mathbb{R}^n}{\max}~U_k(x) > M > 0.
\end{align}
A straightforward calculation shows that
\begin{align}\label{D305}
(-\Delta)^s U_k(\tilde{x}^k)-\Delta U_k(\tilde{x}^k) \geq C_{n,s}P.V.\int_{\mathbb{R}^n} \frac{U_k(\tilde{x}^k) - U_k(y)}{|\tilde{x}^k - y|^{n+2s}}dy  \geq 0.
\end{align}
Combining $c(x)\geq 0$, \eqref{D32}, \eqref{D311}, \eqref{D303} and \eqref{D305}, we obtain
\begin{align}\label{D306}
0\leq (-\Delta)^s U_k(\tilde{x}^k)\leq& (-\Delta)^s U_k(\tilde{x}^k)-\Delta U_k(\tilde{x}^k)\nonumber\\
=&(-\Delta)^s u(\tilde{x}^k) + \varepsilon_k (-\Delta)^s \psi_{r_k}(\tilde{x}^k)-\Delta u(\tilde{x}^k)-\varepsilon_k \Delta \psi_{r_k}(\tilde{x}^k) \nonumber\\
\leq &-c(\tilde{x}^k)u(\tilde{x}^k)+\varepsilon_k\frac{C_1}{r_k^{2s}} +\varepsilon_k\frac{C_2}{r_k^{2}}\nonumber\\
\leq &\varepsilon_k \frac{C}{r_k^{2s}},
\end{align}
where $C_1$, $C_2$ and $C$ are positive constants and we have used the estimates
$$|(-\lap)^s \psi_{r_k}(\tilde{x}^k)|\leq \frac{C_1}{r_k^{2s}}~\mbox{ and }~|-\lap  \psi_{r_k}(\tilde{x}^k)|\leq \frac{C_2}{r_k^2}. $$

Applying \eqref{D304} and Lemma~\ref{L2} to $U_k$ at $\tilde{x}^k$, we infer
\begin{align}\label{D307}
\frac{C_0}{C_{n,s}} r_k^{2s} \{(-\Delta)^s U_k(\tilde{x}^k)-\Delta U_k(\tilde{x}^k)\} +C_0 r_k^{2s} \int_{B_{r_k}^c(\tilde{x}^k)} \frac{U_k(y)}{|\tilde{x}^k - y|^{n+2s}} dy \geq U_k(\tilde{x}^k) \geq M,
\end{align}
where
$$C_0 = \frac{1}{\int_{B_1^c(0)} \frac{1}{|y|^{n+2s}}dy}.$$
On the other hand, it follows from \eqref{D306} that
\begin{align}\label{D308}
\frac{C_0}{C_{n,s}} r_k^{2s} \{(-\Delta)^s U_k(\tilde{x}^k)-\Delta U_k(\tilde{x}^k)\}\leq\frac{C_0}{C_{n,s}}r_k^{2s} \varepsilon_k\frac{C}{r_k^{2s}}\leq \varepsilon_k C_3,
\end{align}
where $C_3$ is a positive constant.

Since $R > R / \sqrt[n]{2}$, applying \eqref{D31} and the result of \cite{CWA}, we obtain
$$\mathop{\underline{\lim}}\limits_{R \to +\infty} \frac{\bigl| \{ B_R(\tilde{x}^k) \setminus B_{R/\sqrt[n]{2}}(\tilde{x}^k) \} \cap \Omega^c \bigr|}{|B_R(\tilde{x}^k)|} > 0.$$
This implies that there exist a positive constant $\tilde{C}$ and a sufficiently large $R_k$ such that
\begin{align}\label{D3081}
\frac{\bigl| \{ B_R(\tilde{x}^k) \setminus B_{R/\sqrt[n]{2}}(\tilde{x}^k) \} \cap \Omega^c \bigr|}{|B_R(\tilde{x}^k)|} \ge \tilde{C} > 0, \quad \text{for }~ R \ge R_k.
\end{align}

Let $r_k = R_k / \sqrt[n]{2}$. Using \eqref{D304}, \eqref{D3081}, and the fact that
$$U_k(y) = u(y) \le 0, \quad y \in B_{r_k}^c(\tilde{x}^k) \cap \Omega^c,$$
a straightforward calculation then yields
\begin{align}\label{D309}
 & C_0 r_k^{2s} \int_{B_{r_k}^c(\tilde{x}^k)} \frac{U_k(y)}{|\tilde{x}^k - y|^{n+2s}} dy \nonumber\\
=& C_0 (R_k / \sqrt[n]{2})^{2s} \int_{B_{R_k/\sqrt[n]{2}}^c(\tilde{x}^k)} \frac{U_k(y)}{|\tilde{x}^k - y|^{n+2s}} dy \nonumber\\
=& C_0 (R_k / \sqrt[n]{2})^{2s} \{ \int_{B_{R_k/\sqrt[n]{2}}^c(\tilde{x}^k) \cap \Omega} \frac{U_k(y)}{|\tilde{x}^k - y|^{n+2s}} dy +\int_{B_{R_k/\sqrt[n]{2}}^c(\tilde{x}^k) \cap \Omega^c} \frac{U_k(y)}{|\tilde{x}^k - y|^{n+2s}} dy\nonumber\\
&+ \int_{B_{R_k/\sqrt[n]{2}}^c(\tilde{x}^k) \cap \Omega^c} \frac{M + \varepsilon_k}{|\tilde{x}^k - y|^{n+2s}} dy- \int_{B_{R_k/\sqrt[n]{2}}^c(\tilde{x}^k) \cap \Omega^c} \frac{M + \varepsilon_k}{|\tilde{x}^k - y|^{n+2s}} dy  \} \nonumber\\
\leq& C_0 (R_k / \sqrt[n]{2})^{2s} \{ \int_{B_{R_k/\sqrt[n]{2}}^c(\tilde{x}^k) \cap \Omega} \frac{M + \varepsilon_k}{|\tilde{x}^k - y|^{n+2s}} dy+ \int_{B_{R_k/\sqrt[n]{2}}^c(\tilde{x}^k) \cap \Omega^c} \frac{M+ \varepsilon_k}{|\tilde{x}^k - y|^{n+2s}} dy\nonumber\\
&- \int_{B_{R_k/\sqrt[n]{2}}^c(\tilde{x}^k) \cap \Omega^c} \frac{M + \varepsilon_k}{|\tilde{x}^k - y|^{n+2s}} dy  \} \nonumber\\
\leq &C_0 (R_k / \sqrt[n]{2})^{2s} \int_{B_{R_k/\sqrt[n]{2}}^c(\tilde{x}^k)} \frac{M + \varepsilon_k}{|\tilde{x}^k - y|^{n+2s}} dy - C_0 (R_k / \sqrt[n]{2})^{2s} \int_{B_{R_k/\sqrt[n]{2}}^c(\tilde{x}^k) \cap \Omega^c} \frac{M + \varepsilon_k}{|\tilde{x}^k - y|^{n+2s}} dy \nonumber\\
=& M + \varepsilon_k - C_0 (R_k / \sqrt[n]{2})^{2s} \int_{B_{R_k/\sqrt[n]{2}}^c(\tilde{x}^k) \cap \Omega^c} \frac{M + \varepsilon_k}{|\tilde{x}^k - y|^{n+2s}} dy \nonumber\\
\leq& M + \varepsilon_k - C_0 (R_k / \sqrt[n]{2})^{2s} \int_{\{B_{R_k}(\tilde{x}^k) \setminus B_{R_k/\sqrt[n]{2}}(\tilde{x}^k)\} \cap \Omega^c} \frac{M + \varepsilon_k}{|\tilde{x}^k - y|^{n+2s}} dy \nonumber\\
\leq& M + \varepsilon_k - (M + \varepsilon_k)C_0 (R_k / \sqrt[n]{2})^{2s} R_k^{-(n+2s)} |\{B_{R_k}(\tilde{x}^k) \setminus B_{R_k/\sqrt[n]{2}}(\tilde{x}^k)\} \cap \Omega^c|\nonumber\\
\leq &M + \varepsilon_k - (M + \varepsilon_k) C_0 (R_k / \sqrt[n]{2})^{2s} R_k^{-(n+2s)} \tilde{C} |B_{R_k}(\tilde{x}^k)| \nonumber\\
=& (1 - C)(M + \varepsilon_k).
\end{align}
Combining $r_k=R_k / \sqrt[n]{2}$, \eqref{D306}, \eqref{D307}, \eqref{D308} and \eqref{D309}, we have
\begin{align*}
M\leq&\frac{C_0}{C_{n,s}} r_k^{2s} \{(-\Delta)^s U_k(\tilde{x}^k)-\Delta U_k(\tilde{x}^k)\} +C_0 r_k^{2s} \int_{B_{r_k}^c(\tilde{x}^k)} \frac{U_k(y)}{|\tilde{x}^k - y|^{n+2s}} dy\\
\leq&C_3 \varepsilon_k+(1 - C)(M + \varepsilon_k),
\end{align*}
which implies
$$M\leq C_4 \varepsilon_k,$$
where $C_4$ is a positive constant. This is a contradiction for sufficiently large $k$. Hence, \eqref{D301} cannot hold.

This proves \eqref{D33} and completes the proof of Theorem~\ref{D3}.
\end{proof}

\subsection{Maximum principles and key lemmas  involving parabolic  mixed local and nonlocal  operators}
\
\newline
\indent
Let
$$u_{\tau}(x,t)= u(x_\tau, t)~\mbox{ and }~w_\tau(x, t) = u(x, t) - u_\tau(x, t).$$

We first present a generalized average inequality.

\begin{lemma}\label{L1}
For each fixed $t \in \mathbb{R}$, assume that
 $u(x, t) \in (C^2(\mathbb R^n)\cap {\tilde {\mathcal L}}_{2s})\times (C^1(\mathbb R))$
attains its maximum at a point $\tilde{x} \in \mathbb{R}^n$. Then, for any $r>0$, the following inequality holds:
\begin{align}\label{L11}
\frac{C_0}{C_{n,s}}\, r^{2s} \bigl\{(-\Delta)^s u(\tilde{x},t) - \Delta u(\tilde{x},t)\bigr\}
+ C_0\, r^{2s} \int_{B_r^c(\tilde{x})} \frac{u(y,t)}{|\tilde{x}-y|^{n+2s}} \, dy
\ge u(\tilde{x},t),
\end{align}
where $B_r^c(\tilde{x})$ denotes the complement of $B_r(\tilde{x})$ in $\mathbb{R}^n$, and
$$C_0 := \frac{1}{\int_{B_1^c(0)} \frac{1}{|y|^{n+2s}}\, dy} > 0,$$
satisfying
$$C_0\, r^{2s} \int_{B_r^c(\tilde{x})} \frac{1}{|\tilde{x}-y|^{n+2s}} \, dy = 1.$$
\end{lemma}

\begin{remark}
Lemma~\ref{L1} can be regarded as an extension of Theorem~1.1 in \cite{CWA}.
\end{remark}
\begin{proof}
The proof follows similarly to that of Lemma~\ref{L2} and Theorem~1.1 in \cite{CWA}.

Suppose that, for a given $t \in \mathbb{R}$, $u(x,t)$ attains its maximum at $\tilde{x}$. By the definitions of the fractional Laplacian and the classical Laplacian, we have
\begin{align*}
(-\Delta)^s u(\tilde{x}, t)-\Delta u(\tilde{x},t)
\geq& (-\Delta)^s u(\tilde{x}, t)\\
=& C_{n,s}P.V. \int_{\mathbb{R}^n} \frac{u(\tilde{x}, t) - u(y, t)}{|\tilde{x} - y|^{n+2s}}dy \\
\geq& C_{n,s} \int_{B_r^c(\tilde{x})} \frac{u(\tilde{x}, t) - u(y, t)}{|\tilde{x} - y|^{n+2s}}dy \\
=& C_{n,s} u(\tilde{x}, t) \int_{B_r^c(\tilde{x})} \frac{1}{|\tilde{x} - y|^{n+2s}}dy - C_{n,s} \int_{B_r^c(\tilde{x})} \frac{u(y, t)}{|\tilde{x} - y|^{n+2s}}dy \\
=& C_{n,s} \frac{u(\tilde{x}, t)}{r^{2s}} \int_{B_1^c(0)} \frac{1}{|y|^{n+2s}}dy - C_{n,s} \int_{B_r^c(\tilde{x})} \frac{u(y, t)}{|\tilde{x} - y|^{n+2s}}dy \\
=& \frac{C_{n,s}}{C_0} \frac{u(\tilde{x}, t)}{r^{2s}} - C_{n,s} \int_{B_r^c(\tilde{x})} \frac{u(y, t)}{|\tilde{x} - y|^{n+2s}}dy,
\end{align*}
where $$C_0 = \frac{1}{\int_{B_1^c(0)} \frac{1}{|y|^{n+2s}}dy}$$ is a positive constant. Thus,
\begin{align*}
\frac{C_0}{C_{n,s}} r^{2s} \{(-\Delta)^s u(\tilde{x}, t)-\Delta u(\tilde{x},t)\} + C_0 r^{2s}\int_{B_r^c(\tilde{x})} \frac{u(y, t)}{{|\tilde{x} - y|^{n+2s}}}dy \geq u(\tilde{x}, t),
\end{align*}
which proves  \eqref{L11}.

This completes the proof of Lemma \ref{L1}.
\end{proof}

\begin{proof}[Proof of Theorem \ref{N1}]

Recall that $\Omega$ is bounded and $w_\tau(x,t) = u(x,t) - u_\tau(x,t)$. Suppose that \eqref{N13} does not hold. Since $\Omega$ is bounded, $U_\tau$ is bounded from below in $\Omega \times \mathbb{R}$, and for any fixed $t \in \mathbb{R}$, $U_\tau(\cdot,t)$ is lower semicontinuous on $\overline{\Omega}$. Then there exist a point $x(t) \in \Omega$ and a constant $\bar m > 0$ such that
\begin{align}\label{N101}
\inf_{(x,t) \in \Omega \times \mathbb{R}} U_\tau(x,t)
= \inf_{t \in \mathbb{R}} U_\tau(x(t),t) := -\bar m < 0.
\end{align}

Consequently, there exists a sequence $\{t_k\} \subset \mathbb{R}$ and a sequence $\{\bar m_k\} \nearrow \bar m$ such that
$$U_\tau(x(t_k), t_k) = -\bar m_k \searrow -\bar m, ~\mbox{ as }~ k \to \infty.$$

Since the infimum of $U_\tau$ with respect to $t$ may not be attained, we perturb $U_\tau$ in time so that the infimum $-\bar m$ is achieved by the perturbed function. Define the auxiliary function
$$U(x,t) := U_\tau(x,t) - \varepsilon_k \gamma_k(t),$$
where $\varepsilon_k := \bar m - \bar m_k$ and $\gamma_k(t) := \gamma(t - t_k)$ with $\gamma \in C_0^\infty(-1,1)$, $0 \le \gamma \le 1$, satisfying
$$
\gamma(t) =
\begin{cases}
1, ~& |t| \le \frac{1}{2},\\
0, ~& |t| \ge 1.
\end{cases}
$$

Clearly, $\mathrm{supp}\, \gamma_k \subset (t_k-1, t_k+1)$ and $\gamma_k(t_k) = 1$.
Using the exterior condition in \eqref{N11} together with \eqref{N101}, we then derive
$$\left\{\begin{array}{ll}
U(x(t_k),t_k)=-\bar m,\\
U(x,t)=U_\tau(x,t)\geq -\bar m,~(x,t)\in \Omega \times [\mathbb R \setminus (-1+t_k,1+t_k)],\\
U(x,t)\geq -\varepsilon_k \gamma_k(t)>-\bar m,~(x,t)\in \Omega^c\times \mathbb R.
\end{array}
\right.$$

Due to the lower semicontinuity of $U_\tau$ in $\overline{\Omega} \times (t_k-1, t_k+1)$, the function $U$ attains its minimum value, which is at most $-\bar m$, in $\Omega \times (t_k-1, t_k+1)$. That is, there exists a sequence
$$
\{(\tilde{x}^k, \tilde{t}_k)\} \subset \Omega \times (t_k-1, t_k+1)
$$
such that
\begin{equation}\label{2233}
-\bar m - \varepsilon_k \le U(\tilde{x}^k, \tilde{t}_k) := \inf_{\Omega \times \mathbb{R}} U(x,t) \le -\bar m.
\end{equation}

Consequently, we have
$$
-\bar m \le U_\tau(\tilde{x}^k, \tilde{t}_k) \le -\bar m_k < 0.
$$

Using the definition of $U$ and \eqref{2233}, we obtain
\begin{equation*}
\begin{cases}
\partial_t U(\tilde{x}^k, \tilde{t}_k) \le 0,\\
\partial_t^\alpha U(\tilde{x}^k, \tilde{t}_k)
= C_\alpha \int_{-\infty}^{\tilde{t}_k} \frac{U(\tilde{x}^k, \tilde{t}_k) - U(\tilde{x}^k, \tau)}{(\tilde{t}_k - \tau)^{1+\alpha}}\, d\tau \le 0, \quad \text{for }~ 0 < \alpha < 1,
\end{cases}
\end{equation*}
and
$$
-\Delta U(\tilde{x}^k, \tilde{t}_k) \le 0.
$$
Moreover,
$$\aligned (-\lap)^s U(\tilde x^k,\tilde t_k)&=C_{n,s}P.V.\int_{\mathbb R^n} \frac{U(\tilde x^k,\tilde t_k)-U( y,\tilde t_k)}{|\tilde x^k-y|^{n+2s}}dy\\
&=C_{n,s}P.V.\int_{\Omega} \frac{U(\tilde x^k,\tilde t_k)-U( y,\tilde t_k)}{|\tilde x^k-y|^{n+2s}}dy+C_{n,s}P.V.\int_{\Omega^c} \frac{U(\tilde x^k,\tilde t_k)-U( y,\tilde t_k)}{|\tilde x^k-y|^{n+2s}}dy\\
&\leq C_{n,s} \int_{\Omega^c} \frac{U(\tilde x^k,\tilde t_k) }{|\tilde x^k-y|^{n+2s}}dy\\
&=C_{n,s} \int_{\Omega^c} \frac{U_\tau(\tilde x^k,\tilde t_k) -\varepsilon_k \gamma_k(\tilde t)}{|\tilde x^k-y|^{n+2s}}dy\\
&\leq U_\tau(\tilde x^k,\tilde t_k) \frac{C}{l^{2s}_n(\Omega)}.
\endaligned $$
Hence, we have
\begin{equation}\label{1133}
  {\partial^\alpha_ t}U(\tilde{x}^k, \tilde{t}_k) +(-\lap)^s U(\tilde x^k,\tilde t_k)-   \lap U(\tilde x^k,\tilde t_k)\leq U_\tau(\tilde x^k,\tilde t_k) \frac{C}{l^{2s}_n(\Omega)} .
\end{equation}

On the other hand, by   \eqref{N11}, we obtain
\begin{equation}\label{11}\aligned
 &{\partial^\alpha_ t}U(\tilde{x}^k, \tilde{t}_k) +(-\lap)^s U(\tilde x^k,\tilde t_k)-   \lap U(\tilde x^k,\tilde t_k)\\
 =&{\partial^\alpha_ t}[U_\tau(\tilde{x}^k, \tilde{t}_k) -\varepsilon_k \gamma_k(\tilde t)]+(-\lap)^s U_\tau(\tilde x^k,\tilde t_k)-   \lap U_\tau(\tilde x^k,\tilde t_k)\\
 =& {\partial^\alpha_ t}U_\tau(\tilde{x}^k, \tilde{t}_k) +(-\lap)^s U_\tau(\tilde x^k,\tilde t_k)-   \lap U_\tau(\tilde x^k,\tilde t_k)-\varepsilon_k {\partial^\alpha_ t}\gamma_k(\tilde t_k)\\
 \geq & -C(\tilde x^k,\tilde t_k)U_\tau(\tilde x^k,\tilde t_k)-\varepsilon_k {\partial^\alpha_ t}\gamma_k(\tilde t_k)\\
 \geq & -C(\tilde x^k,\tilde t_k)U_\tau(\tilde x^k,\tilde t_k)-C_1 \varepsilon_k,
\endaligned \end{equation}
where we have used the fact that
$$| {\partial^\alpha_ t}\gamma_k(\tilde t_k)|\leq C_1$$ with  $C_1$ is a positive constant.
Combining  \eqref{1133} and \eqref{11}, we deduce
$$-C(\tilde x^k,\tilde t_k)U_\tau(\tilde x^k,\tilde t_k)-C_1 \varepsilon_k \leq U_\tau(\tilde x^k,\tilde t_k) \frac{C}{l^{2s}_n(\Omega)},$$
which implies
\begin{align*}
0&\leq U_\tau(\tilde x^k,\tilde t_k)\{\frac{C}{l^{2s}_n(\Omega)}+C(\tilde x^k,\tilde t_k) \} + C_1\varepsilon_k \\
&\leq U_\tau(\tilde x^k,\tilde t_k)\{\frac{C}{l^{2s}_n(\Omega)}+\underset{\Omega \times \mathbb R}{\inf}C(x,t) \} + C_1\varepsilon_k.
\end{align*}
By condition \eqref{1122}, the above inequality leads to a contradiction as $k \to \infty$.

Therefore, \eqref{N13} holds. This completes the proof.
\end{proof}

\begin{proof}[Proof of Theorem \ref{T2}]
Suppose that \eqref{T23} is false. Since $u(x,t)$ is bounded from above in $D \times \mathbb{R}$, there exists a positive constant $M$ such that
\begin{align}\label{T201}
\sup_{D \times \mathbb{R}} u(x,t) := M > 0.
\end{align}

Because the set $D \times \mathbb{R}$ is unbounded, the supremum of $u(x,t)$ may not be attained. Nevertheless, by \eqref{T201}, there exists a sequence $\{(x^k, t_k)\} \subset D \times \mathbb{R}$ such that
\begin{align}\label{1T202}
u(x^k, t_k) \to M \quad \text{as }~ k \to \infty.
\end{align}
More precisely, there exists a nonnegative sequence $\{\varepsilon_k\} \searrow 0$ satisfying
\begin{align}\label{T202}
u(x^k, t_k) = M - \varepsilon_k > 0.
\end{align}

Next, we claim that $\{x^k\}$ is uniformly bounded away from $D^c$; that is, there exists a positive constant $c$ such that
$$
\mathrm{dist}(x^k, D^c) \ge c > 0.
$$
Without loss of generality, we may assume $c \ge 2$.

Indeed, if this were not the case, $\{x^k\}$ would converge to $D^c$. Since $u(x,t) \le 0$ in $D^c \times \mathbb{R}$, we would have
$$
u(x^k, t_k) \to 0 \quad \text{as }~ k \to \infty,
$$
contradicting \eqref{1T202}. Therefore, we may assume that
\begin{align}\label{001}
\mathrm{dist}(x^k, D^c) \ge c > 0.
\end{align}

We now introduce the following auxiliary function:
\begin{align}\label{T203}
V_k(x, t) := u(x, t) + \varepsilon_k \zeta_k(x, t),
\end{align}
where
$$
\zeta_k(x,t) = \zeta\left( \frac{x - x^k}{r_k}, \frac{t - t_k}{r_k^{\frac{2s}{\alpha}}} \right),
$$
with any fixed
$$
r_k = \frac{1}{2} \mathrm{dist}(x^k, D^c) \ge \frac{1}{2} c = \bar c > 0,
$$
and $\zeta(x,t) \in C_0^\infty(\mathbb{R}^n \times \mathbb{R})$ satisfying
\begin{equation*}
\begin{cases}
0 \le \zeta(x,t) \le 1, ~& \text{in }~ \mathbb{R}^n \times \mathbb{R},\\
\zeta(x,t) = 1, ~& \text{in }~ B_{\frac{1}{2}}(0) \times [-\frac{1}{2}, \frac{1}{2}],\\
\zeta(x,t) = 0, ~& \text{in }~ (\mathbb{R}^n \times \mathbb{R}) \setminus (B_{\frac{1}{2}}(0) \times [-1,1]).
\end{cases}
\end{equation*}

We define the parabolic cylinder centered at $(x^k, t_k)$ by
$$
E_{r_k}(x^k, t_k) := B_{r_k}(x^k) \times [t_k - r_k^{\frac{2s}{\alpha}},\, t_k + r_k^{\frac{2s}{\alpha}}] \subset D \times \mathbb{R}.
$$
For the mixed local and nonlocal parabolic equation with $\alpha = 1$, the parameter $\alpha$ in the definitions of $\zeta_k(x,t)$ and $E_{r_k}(x^k,t_k)$ can also be taken as 1.

By \eqref{T202} and \eqref{T203}, we have
$$
V_k(x^k, t_k) = M - \varepsilon_k + \varepsilon_k = M > 0,
$$
and
$$
V_k(x,t) \le M, \quad (x,t) \in (\mathbb{R}^n \times \mathbb{R}) \setminus E_{r_k}(x^k, t_k).
$$

This implies that the maximum of $V_k(x,t)$ over $E_{r_k}(x^k, t_k)$ is at least as large as its values over the complement, and thus the global maximum of $V_k(x,t)$ over $\mathbb{R}^n \times \mathbb{R}$ is attained within $E_{r_k}(x^k, t_k)$. Consequently, there exists a point
$$
(\tilde{x}^k, \tilde{t}_k) \in \overline{E_{r_k}(x^k, t_k)} \subset D \times \mathbb{R}
$$
such that
\begin{align}\label{T204}
M + \varepsilon_k \ge V_k(\tilde{x}^k, \tilde{t}_k) = \sup_{\mathbb{R}^n \times \mathbb{R}} V_k(x,t) \ge M > 0.
\end{align}
It follows that
\begin{equation}\label{T205}
\left\{
\begin{array}{ll}
{\partial_ t}V_k(\tilde{x}^k, \tilde{t}_k)\geq 0,\\
{\partial^\alpha_ t}V_k(\tilde{x}^k, \tilde{t}_k)=C_\alpha \int_{-\infty}^{\tilde{t}_k} \frac{V_k(\tilde{x}^k, \tilde{t}_k)-V_k(\tilde{x}^k, \tau)}{(\tilde{t}_k-\tau)^{1+\alpha}}d\tau \geq 0,~\mbox{ if }~0<\alpha<1.
 \end{array}
\right.
\end{equation}
Obviously,
\begin{align}\label{T206}
(-\Delta)^s V_k(\tilde{x}^k, \tilde{t}_k)-\Delta V_k(\tilde{x}^k, \tilde{t}_k)\geq C_{n,s}P.V.\int_{\mathbb{R}^n} \frac{V_k(\tilde{x}^k, \tilde{t}_k) - V_k(y, \tilde{t}_k)}{|\tilde{x}^k - y|^{n+2s}}dy \geq 0.
\end{align}
By \eqref{T203} and \eqref{T205}, we  obtain
$${\partial^\alpha_ t}V_k(\tilde{x}^k, \tilde{t}_k)={\partial^\alpha_ t}u(\tilde{x}^k, \tilde{t}_k)+ \varepsilon_k {\partial^\alpha_ t}\zeta_k(\tilde{x}^k, \tilde{t}_k)\geq 0,$$
which implies
\begin{align}\label{T207}
-{\partial^\alpha_ t}u(\tilde{x}^k, \tilde{t}_k)\leq \varepsilon_k {\partial^\alpha_ t}\zeta_k(\tilde{x}^k, \tilde{t}_k)\leq \frac{C_1 \varepsilon_k}{r_k^{2s}},
\end{align}
where $C_1$ is a positive constant.  Combining \eqref{T22}, \eqref{001}, \eqref{T203}, \eqref{T206}, \eqref{T207} and
$$|(-\Delta)^s\zeta_k(\tilde{x}^k, \tilde{t}_k)|\leq \frac{C_2}{r_k^{2s}},~|-\Delta \zeta_k(\tilde{x}^k, \tilde{t}_k)|\leq \frac{C_3}{r_k^2},$$  we conclude that
\begin{align}\label{T208}
0 \leq& (-\Delta)^s V_k(\tilde{x}^k, \tilde{t}_k)-\Delta V_k(\tilde{x}^k, \tilde{t}_k)\nonumber\\
 =& (-\Delta)^s u(\tilde{x}^k, \tilde{t}_k) + \varepsilon_k (-\Delta)^s \zeta_k(\tilde{x}^k, \tilde{t}_k)-\Delta u(\tilde{x}^k, \tilde{t}_k)-\varepsilon_k \Delta \zeta_k(\tilde{x}^k, \tilde{t}_k) \nonumber\\
 \leq &-\partial^\alpha_t u(\tilde{x}^k, \tilde{t}_k) + \varepsilon_k (-\Delta)^s \zeta_k(\tilde{x}^k, \tilde{t}_k)-\varepsilon_k \Delta \zeta_k(\tilde{x}^k, \tilde{t}_k) \nonumber\\
\leq& \frac{C_1 \varepsilon_k}{r_k^{2s}} + \frac{C_2 \varepsilon_k}{r_k^{2s}}+\frac{C_3 \varepsilon_k}{r_k^{2}} \nonumber\\
\leq &\frac{C_4 \varepsilon_k}{r_k^{2s}},
\end{align}
where $C_2$, $C_3$, and $C_4$ are positive constants, and
$$
r_k=\tfrac{1}{2}\,\mathrm{dist}(x^k,D^c)\geq \bar c>0 .
$$
Applying \eqref{T204} and Lemma~\ref{L1} to $V_k$ at the point
$(\tilde{x}^k,\tilde{t}_k)$, for any $r_k\geq \bar c>0$, we infer that
\begin{align}\label{T209}
\frac{C_0}{C_{n,s}} r_k^{2s}
\bigl\{(-\Delta)^s V_k(\tilde{x}^k, \tilde{t}_k)
-\Delta V_k(\tilde{x}^k, \tilde{t}_k)\bigr\}
+ C_0 r_k^{2s}
\int_{B_{r_k}^c(\tilde{x}^k)}
\frac{V_k(y, \tilde{t}_k)}{|\tilde{x}^k-y|^{n+2s}}\,dy
\geq V_k(\tilde{x}^k, \tilde{t}_k)
\geq M ,
\end{align}
where
$$
C_0=\frac{1}{\displaystyle\int_{B_1^c(0)}\frac{1}{|y|^{n+2s}}\,dy}.
$$

To derive a contradiction, we show that the left-hand side of
\eqref{T209} is strictly less than $M$.
It follows from \eqref{T208} that
\begin{align}\label{T210}
\frac{C_0}{C_{n,s}} r_k^{2s}
\bigl\{(-\Delta)^s V_k(\tilde{x}^k, \tilde{t}_k)
-\Delta V_k(\tilde{x}^k, \tilde{t}_k)\bigr\}
&\leq
\frac{C_0}{C_{n,s}} r_k^{2s}\frac{C_4\varepsilon_k}{r_k^{2s}}  \\
&\leq C_5\,\varepsilon_k ,
\end{align}
where $C_5$ is a positive constant.
Hence, for sufficiently large $k$, the first term on the left-hand side
of \eqref{T209} is arbitrarily small.

It remains to estimate the second term in \eqref{T209}, namely
$$
C_0 r_k^{2s}
\int_{B_{r_k}^c(\tilde{x}^k)}
\frac{V_k(y, \tilde{t}_k)}{|\tilde{x}^k-y|^{n+2s}}\,dy .
$$
Since $R>R/\sqrt[n]{2}$, by \eqref{T21} and the result of \cite{CWA}, we have
$$
\mathop{\underline{\lim}}\limits_{R\to+\infty}
\frac{\bigl|\{B_R(\tilde{x}^k)\setminus B_{R/\sqrt[n]{2}}(\tilde{x}^k)\}
\cap \Omega^c\bigr|}
{|B_R(\tilde{x}^k)|}
>0 .
$$
Consequently, there exist a positive constant $\tilde{C}$ and
sufficiently large $R_k$ such that
\begin{align}\label{T211}
\frac{|\{B_R(\tilde{x}^k) \setminus B_{R/\sqrt[n]{2}}(\tilde{x}^k)\} \cap \Omega^c|}{|B_R(\tilde{x}^k)|} \geq \tilde{C} > 0,~\mbox{ for }~R \geq R_k.
\end{align}
Let $r_k = R_k / \sqrt[n]{2}$, by \eqref{T204}, \eqref{T211} and  the fact
$$V_k(y, \tilde{t}_k) = u(y, \tilde{t}_k) \leq 0,~ y \in B_{r_k}^c(\tilde{x}^k) \cap \Omega^c,$$
we obtain
\begin{align}\label{T212}
& C_0 r_k^{2s} \int_{B_{r_k}^c(\tilde{x}^k)} \frac{V_k(y, \tilde{t}_k)}{|\tilde{x}^k - y|^{n+2s}} dy \nonumber\\
=& C_0 (R_k / \sqrt[n]{2})^{2s} \int_{B_{R_k/\sqrt[n]{2}}^c(\tilde{x}^k)} \frac{V_k(y, \tilde{t}_k)}{|\tilde{x}^k - y|^{n+2s}} dy \nonumber\\
=& C_0 (R_k / \sqrt[n]{2})^{2s} \{ \int_{B_{R_k/\sqrt[n]{2}}^c(\tilde{x}^k) \cap \Omega} \frac{V_k(y, \tilde{t}_k)}{|\tilde{x}^k - y|^{n+2s}} dy +\int_{B_{R_k/\sqrt[n]{2}}^c(\tilde{x}^k) \cap \Omega^c} \frac{V_k(y, \tilde{t}_k)}{|\tilde{x}^k - y|^{n+2s}} dy\nonumber\\
&+ \int_{B_{R_k/\sqrt[n]{2}}^c(\tilde{x}^k) \cap \Omega^c} \frac{M + \varepsilon_k}{|\tilde{x}^k - y|^{n+2s}} dy- \int_{B_{R_k/\sqrt[n]{2}}^c(\tilde{x}^k) \cap \Omega^c} \frac{M + \varepsilon_k}{|\tilde{x}^k - y|^{n+2s}} dy  \} \nonumber\\
\leq& C_0 (R_k / \sqrt[n]{2})^{2s} \{ \int_{B_{R_k/\sqrt[n]{2}}^c(\tilde{x}^k) \cap \Omega} \frac{M + \varepsilon_k}{|\tilde{x}^k - y|^{n+2s}} dy+ \int_{B_{R_k/\sqrt[n]{2}}^c(\tilde{x}^k) \cap \Omega^c} \frac{M+ \varepsilon_k}{|\tilde{x}^k - y|^{n+2s}} dy\nonumber\\
&- \int_{B_{R_k/\sqrt[n]{2}}^c(\tilde{x}^k) \cap \Omega^c} \frac{M + \varepsilon_k}{|\tilde{x}^k - y|^{n+2s}} dy  \} \nonumber\\
\leq &C_0 (R_k / \sqrt[n]{2})^{2s} \int_{B_{R_k/\sqrt[n]{2}}^c(\tilde{x}^k)} \frac{M + \varepsilon_k}{|\tilde{x}^k - y|^{n+2s}} dy - C_0 (R_k / \sqrt[n]{2})^{2s} \int_{B_{R_k/\sqrt[n]{2}}^c(\tilde{x}^k) \cap \Omega^c} \frac{M + \varepsilon_k}{|\tilde{x}^k - y|^{n+2s}} dy \nonumber\\
=& M + \varepsilon_k - C_0 (R_k / \sqrt[n]{2})^{2s} \int_{B_{R_k/\sqrt[n]{2}}^c(\tilde{x}^k) \cap \Omega^c} \frac{M + \varepsilon_k}{|\tilde{x}^k - y|^{n+2s}} dy \nonumber\\
\leq& M + \varepsilon_k - C_0 (R_k / \sqrt[n]{2})^{2s} \int_{\{B_{R_k}(\tilde{x}^k) \setminus B_{R_k/\sqrt[n]{2}}(\tilde{x}^k)\} \cap \Omega^c} \frac{M + \varepsilon_k}{|\tilde{x}^k - y|^{n+2s}} dy \nonumber\\
\leq& M + \varepsilon_k - (M + \varepsilon_k)C_0 (R_k / \sqrt[n]{2})^{2s} R_k^{-(n+2s)} |\{B_{R_k}(\tilde{x}^k) \setminus B_{R_k/\sqrt[n]{2}}(\tilde{x}^k)\} \cap \Omega^c|\nonumber\\
\leq &M + \varepsilon_k - (M + \varepsilon_k) C_0 (R_k / \sqrt[n]{2})^{2s} R_k^{-(n+2s)} \tilde{C} |B_{R_k}(\tilde{x}^k)| \nonumber\\
=& (1 - \bar C)(M + \varepsilon_k),
\end{align}
where $\bar C$ is a positive constant. Combining $r_k=R_k / \sqrt[n]{2}$, \eqref{T208}, \eqref{T209}, \eqref{T210} and \eqref{T212}, one has
\begin{align*}
M\leq&\frac{C_0}{C_{n,s}} r_k^{2s} \{(-\Delta)^s V_k(\tilde{x}^k, \tilde{t}_k)-\Delta V_k(\tilde{x}^k, \tilde{t}_k)\} +C_0 r_k^{2s} \int_{B_r^c(\tilde{x}^k)} \frac{V_k(y, \tilde{t}_k)}{|\tilde{x}^k - y|^{n+2s}} dy\\
\leq&C_5 \varepsilon_k+(1 - \bar C)(M + \varepsilon_k),
\end{align*}
furthermore,
$$(C_5-\bar C+1)\varepsilon_k\geq \bar C M,$$
that is, $$\varepsilon_k\geq C_6 M>0,$$
where $C_5$ and $C_6$ are positive constants.
This is a contradiction when $k$ is sufficiently large. Therefore, the assumption \eqref{T201} is incorrect.

Hence, this verifies \eqref{T23} and  the proof of Theorem \ref{T2} is completed.
\end{proof}

\section{Monotonicity and one-dimensional symmetry for  elliptic equations}

In this section, we shall prove Theorem \ref{D2}, Theorem \ref{D4} and Theorem \ref{D5}.

\subsection{Monotonicity in bounded domains}

\begin{proof}[Proof of Theorem \ref{D2}]
For $\tau \ge 0$, let $u_\tau(x)=u(x',x_n+\tau)$.
Then $u_\tau$ is defined in $\Omega_\tau:=\Omega-\tau e_n$, which is obtained by sliding $\Omega$ downward by $\tau$ along the $x_n$-axis, where $e_n=(0',1)$.

\begin{center}
\begin{tikzpicture}[scale=1]
    \draw[->][thick] (-5.5,0) -- (6,0) node at (6.5,0) {$x_n$};
    \draw [red](2.0,0) ellipse (2.5 and 1.5);
   \draw [blue] (-1.8,0) ellipse (2.5 and 1.5);
   \node at (4,-0.4) {$\Omega$};
   \node at (-3.8,-0.4) {$\Omega_\tau$};
   \draw  (0,-0.5) -- (-0.5,-1.7);
    \node at (-0.7,-1.8) {$F_\tau$};
    \draw[->][red] (-1,1.8) -- (1,1.8);
     \draw[dashed, line width=1.5pt] (0.3,0.3) -- (4.1,0.3);
    \fill[black] (0.3,0.3) circle (2pt);
    \fill[black] (4.1,0.3) circle (2pt);
    \node at (0.2,0.6) {$x$};
   \node at (4,0.6) {$x_\tau$};
    \begin{scope}
     \clip (0.1,0) ellipse (0.5 and 0.9);
     \foreach \y in {-0.85,-0.75,...,0.95}
        \draw [dashed][gray][very thin](-0.4,\y) -- (0.7,\y);
    \end{scope}
\node [below=3.0cm, align=flush center,text width=12cm] at (0,0.5)
        {$Figure$ 1. The domain $F_{\tau}$.};
\end{tikzpicture}
\end{center}

Let
$$
F_\tau := \Omega_\tau \cap \Omega \qquad \text{(see Figure~1)}.
$$
Define
$$
\bar{\tau} := \sup\{\tau>0 \mid F_\tau \neq \emptyset\},
$$
and
$$
w_\tau(x) := u(x) - u_\tau(x), \qquad x \in F_\tau,
$$
where $u_\tau$ satisfies the same equation \eqref{D21} in $\Omega_\tau$ as $u$ does in $\Omega$.
Then $w_\tau$ satisfies
\begin{align}\label{D201}
(-\Delta)^s w_\tau(x) - \Delta w_\tau(x) - c(x) w_\tau(x) = 0,
\qquad x \in F_\tau,
\end{align}
where
$$
c(x) := \frac{g(u(x)) - g(u_\tau(x))}{u(x) - u_\tau(x)}.
$$
By the Lipschitz continuity of $g$, the function $c(x)$ belongs to $L^\infty(F_\tau)$ and satisfies
$$
c(x) \le C \qquad \text{ for all }~ x \in F_\tau.
$$
Moreover, from the exterior condition \textbf{(A)}, we have
\begin{align}\label{D2011}
w_\tau(x) \le 0, \qquad x \in F_\tau^{\,c}.
\end{align}

The key step of the proof is to establish that
\begin{align}\label{D202}
w_\tau(x) < 0, \qquad x \in F_\tau, \quad \text{ for all }~ 0 < \tau < \bar{\tau},
\end{align}
which implies that $u$ is strictly increasing in the $x_n$-direction.

Next, we prove \eqref{D202} in two steps.

\medskip
\noindent\textbf{Step~1.}
We first show that
\begin{align}\label{D203}
w_\tau(x) \le 0 \qquad \text{for } \tau \text{ sufficiently close to } \bar{\tau},
\end{align}
when the set $F_\tau$ is narrow.
Combining \eqref{D201} and \eqref{D2011}, and applying Theorem~\ref{D1}, we obtain \eqref{D203}.

\medskip
\noindent\textbf{Step~2.}
Inequality \eqref{D203} provides a starting position from which the sliding argument can be carried out.
We then decrease $\tau$ as long as \eqref{D203} holds, until reaching a limiting position.
Define
$$
\tau_0 := \inf\left\{ \tau \in (0,\bar{\tau}) \,\big|\, w_\tau(x) \le 0 \text{ in } F_\tau \right\}.
$$
We claim that
$$
\tau_0 = 0.
$$

Suppose, by contradiction, that $\tau_0 > 0$.
We will show that the domain $\Omega_\tau$ can be slid upward slightly further while still preserving the inequality
\begin{align}\label{D204}
w_\tau(x) \le 0 \qquad \text{in }~ F_\tau,
\quad \text{for all }~ \tau_0 - \varepsilon < \tau \le \tau_0,
\end{align}
for some $\varepsilon>0$.
This contradicts the definition of $\tau_0$.

By assumption \textbf{(A)}, we have
$$
w_{\tau_0}(x) < 0 \qquad \text{for } x \in \Omega \cap \partial F_{\tau_0}.
$$
Together with
$$
w_{\tau_0}(x) \le 0 \qquad \text{in }~ F_{\tau_0},
$$
it follows that
$$
w_{\tau_0} \not\equiv 0 \qquad \text{in }~ F_{\tau_0}.
$$
If there exists a point $\tilde{x} \in F_{\tau_0}$ such that
$$
w_{\tau_0}(\tilde{x}) = 0,
$$
then $\tilde{x}$ is a maximum point of $w_{\tau_0}$.
A direct computation yields that
\begin{align*}
(-\Delta)^s w_{\tau_0}(\tilde{x})-\Delta w_{\tau_0}(\tilde{x})\geq& (-\Delta)^s w_{\tau_0}(\tilde{x})\\
=&C_{n,s}P.V. \int_{\mathbb{R}^n} \frac{w_{\tau_0}(\tilde{x}) - w_{\tau_0}(y)}{|\tilde{x} - y|^{n+2s}} dy\\
=&C_{n,s}P.V. \{\int_{F_{\tau_0}} \frac{  - w_{\tau_0}(y)}{|\tilde{x} - y|^{n+2s}} dy+\int_{\mathbb{R}^n \setminus F_{\tau_0}} \frac{  - w_{\tau_0}(y)}{|\tilde{x} - y|^{n+2s}} dy\}\\
\geq&C_{n,s}P.V. \int_{F_{\tau_0}} \frac{  - w_{\tau_0}(y)}{|\tilde{x} - y|^{n+2s}} dy\\
>&0,
\end{align*}
which contradicts the identity
\begin{align*}
(-\Delta)^s w_{\tau_0}(\tilde{x}) - \Delta w_{\tau_0}(\tilde{x})
= g(u(\tilde{x})) - g(u_{\tau_0}(\tilde{x})) = 0,
\end{align*}
which follows from the first inequality in \eqref{D21}.
Therefore,
\begin{align}\label{D205}
w_{\tau_0}(x) < 0 \qquad \text{ for all }~ x \in F_{\tau_0}.
\end{align}

Next, we carve out from $F_{\tau_0}$ a closed set $\mathcal{D} \subset F_{\tau_0}$ such that
$F_{\tau_0} \setminus \mathcal{D}$ is narrow.
By \eqref{D205}, there exists a positive constant $\hat C_0$ such that
\begin{align*}
w_{\tau_0}(x) \le -\hat C_0 < 0, \qquad x \in \mathcal{D}.
\end{align*}
By the continuity of $w_\tau$ with respect to $\tau$, for sufficiently small $\varepsilon>0$ we obtain
\begin{align*}
w_{\tau_0-\varepsilon}(x) \le 0, \qquad x \in \mathcal{D}.
\end{align*}
Moreover, by the exterior condition \textbf{(A)}, we have
\begin{align*}
w_{\tau_0-\varepsilon}(x) \le 0, \qquad x \in (F_{\tau_0-\varepsilon})^{c}.
\end{align*}

Since
$$
(F_{\tau_0-\varepsilon} \setminus \mathcal{D})^{c}
= \mathcal{D} \cup (F_{\tau_0-\varepsilon})^{c},
$$
it follows that
\begin{equation*}
\left\{
\begin{array}{ll}
(-\Delta)^s w_{\tau_0-\varepsilon}(x)
- \Delta w_{\tau_0-\varepsilon}(x)
- c_{\tau_0-\varepsilon}(x)\, w_{\tau_0-\varepsilon}(x) = 0,
& x \in F_{\tau_0-\varepsilon} \setminus \mathcal{D}, \\
w_{\tau_0-\varepsilon}(x) \le 0,
& x \in (F_{\tau_0-\varepsilon} \setminus \mathcal{D})^{c}.
\end{array}
\right.
\end{equation*}
Applying Theorem~\ref{D1}, we conclude that \eqref{D204} holds.
This contradicts the definition of $\tau_0$.
Consequently,
\begin{align}\label{07}
w_\tau(x) \le 0 \qquad \text{for all }~ x \in F_\tau,
\quad 0 < \tau < \bar{\tau}.
\end{align}

We now show that \eqref{D202} holds.
By \eqref{07} and the fact that
\begin{align*}
w_\tau \not\equiv 0 \qquad \text{in }~ F_\tau,
\quad 0 < \tau < \bar{\tau},
\end{align*}
if there exists a point $\bar x \in F_\tau$ such that $w_\tau(\bar x)=0$,
then $\bar x$ is a maximum point of $w_\tau$.
Consequently,
$$
(-\Delta)^s w_\tau(\bar x) - \Delta w_\tau(\bar x)
\ge (-\Delta)^s w_\tau(\bar x)
= C_{n,s}\,\mathrm{P.V.}\!\int_{\mathbb{R}^n}
\frac{w_\tau(\bar x)-w_\tau(y)}{|\bar x-y|^{n+2s}}\,dy > 0,
$$
which contradicts
$$
(-\Delta)^s w_\tau(\bar x) - \Delta w_\tau(\bar x)
= g(u(\bar x)) - g(u_\tau(\bar x)) = 0.
$$
Therefore, \eqref{D202} is proved.

This completes the proof of Theorem~\ref{D2}.
\end{proof}

\subsection{Monotonicity in $\mathbb{R}_+^n$}

\begin{proof}[Proof of Theorem~\ref{D4}]
Since $g(u)$ is monotone decreasing in $u$, we have
$$
g(u(x)) - g(u_\tau(x)) \le 0
\qquad \text{at any point } x \in \mathbb{R}^n_+ \text{ where } w_\tau(x) > 0.
$$
It follows from the first equation in \eqref{D41} that
$$
(-\Delta)^s w_\tau(x) - \Delta w_\tau(x) \le 0
\qquad \text{at any point } x \in \mathbb{R}^n_+ \text{ where } w_\tau(x) > 0.
$$

Moreover, the exterior condition
$$
w_\tau(x) \le 0, \qquad x \in \mathbb{R}^n \setminus \mathbb{R}^n_+,
$$
holds trivially.
Applying the maximum principle in unbounded domains (Theorem~\ref{D3}) with
$\Omega = \mathbb{R}^n_+$, we obtain
\begin{align}\label{D402}
w_\tau(x) \le 0 \qquad \text{ for all } x \in \mathbb{R}^n_+,
\quad \text{ for every } \tau > 0.
\end{align}

We now strengthen \eqref{D402} to show that
\begin{align}\label{D401}
w_\tau(x) < 0 \qquad \text{ for all } x \in \mathbb{R}^n_+,
\quad \text{ for every } \tau > 0.
\end{align}
Suppose, by contradiction, that there exist $x_0 \in \mathbb{R}^n_+$ and $\tau_0 > 0$ such that
\begin{align}\label{D403}
w_{\tau_0}(x_0) = \max_{\mathbb{R}^n} w_{\tau_0}(x) = 0.
\end{align}
Since $w_{\tau_0} \not\equiv 0$ in $\mathbb{R}^n$, it follows from
\eqref{D402} and \eqref{D403} that
\begin{align}\label{D404}
(-\Delta)^s w_{\tau_0}(x_0) - \Delta w_{\tau_0}(x_0)
\ge (-\Delta)^s w_{\tau_0}(x_0)
= C_{n,s}\,\mathrm{P.V.}\!\int_{\mathbb{R}^n}
\frac{-w_{\tau_0}(y)}{|x_0 - y|^{n+2s}}\,dy > 0.
\end{align}

On the other hand, by the first equation in \eqref{D41} and \eqref{D403}, we have
\begin{align*}
(-\Delta)^s w_{\tau_0}(x_0) - \Delta w_{\tau_0}(x_0)
&= (-\Delta)^s u(x_0) - \Delta u(x_0)
   - \big[(-\Delta)^s u_{\tau_0}(x_0) - \Delta u_{\tau_0}(x_0)\big] \\
&= g(u(x_0)) - g(u_{\tau_0}(x_0)) \\
&= 0,
\end{align*}
which contradicts \eqref{D404}.
Therefore, \eqref{D401} holds.

This completes the proof of Theorem~\ref{D4}.
\end{proof}

\subsection{One-dimensional symmetry in $\mathbb{R}^n$}

\begin{proof}[Proof of Theorem \ref{D5}]
The proof of Theorem~\ref{D5} is divided into three steps.

\medskip
\noindent\textbf{Step~1.}
We first show that for $\tau$ sufficiently large,
\begin{align}\label{D501}
w_\tau(x) := u(x) - u_\tau(x) \le 0,
\qquad x \in \mathbb{R}^n.
\end{align}
Suppose that \eqref{D501} does not hold.
Then there exists a constant $M>0$ such that
\begin{align}\label{D502}
\sup_{\mathbb{R}^n} w_\tau(x) = M > 0.
\end{align}
To derive a contradiction with \eqref{D502}, we introduce the auxiliary function
$$
\widetilde W_\tau(x) := w_\tau(x) - \frac{M}{2}.
$$

By assumption \eqref{D52}, there exists a sufficiently large constant $\lambda>0$ such that
\begin{align}\label{5D1}
|u(x',x_n)| \ge 1 - \gamma,
\qquad x' \in \mathbb{R}^{n-1}, \ |x_n| \ge \lambda.
\end{align}
Equivalently,
$$
u(x',x_n) \ge 1 - \gamma \quad \text{for } x_n \ge \lambda,
\qquad
u(x',x_n) \le -1 + \gamma \quad \text{for }~ x_n \le -\lambda.
$$
Moreover, again by \eqref{D52}, we may choose a constant $A>\lambda$ such that
\begin{align}\label{D505}
\widetilde W_\tau(x) \le 0,
\qquad x' \in \mathbb{R}^{n-1}, \ x_n \ge A.
\end{align}

We now apply the maximum principle in unbounded domains (Theorem~\ref{D3}) to prove that
\begin{align}\label{q1}
\widetilde W_\tau(x) \le 0,
\qquad x \in \mathbb{R}^n.
\end{align}
Let
$$
\mathcal{H} := \mathbb{R}^{n-1} \times (-\infty, A).
$$
From \eqref{D505}, it follows that
\begin{align}\label{5D5}
\tilde{W}_\tau(x) \leq 0,~ x \in \mathcal{H}^c.
\end{align}
This shows that $\widetilde W_\tau$ satisfies the exterior condition of
Theorem~\ref{D3}.
We next verify the differential inequality satisfied by $\widetilde W_\tau$
in the domain $\mathcal H$.
By \eqref{D51}, \eqref{D54}, and the definition of $\widetilde W_\tau$, we obtain
\begin{align}\label{5D6}
(-\Delta)^s \widetilde W_\tau(x) - \Delta \widetilde W_\tau(x)
&= (-\Delta)^s w_\tau(x) - \Delta w_\tau(x) \nonumber \\
&= g(x,u(x)) - g(x',x_n+\tau,u_\tau(x)) \nonumber \\
&\le g(x,u(x)) - g(x,u_\tau(x)).
\end{align}

We now distinguish three cases:
$$
|x_n| \le \lambda, \qquad x_n > \lambda, \qquad x_n < -\lambda.
$$

\medskip
\noindent\textit{Case~1.} $|x_n| \le \lambda$.

If $|x_n| \le \lambda$ and $\tau \ge 2\lambda$, then $x_n+\tau \ge \lambda$.
By \eqref{5D1}, for any $x' \in \mathbb{R}^{n-1}$ we have
$$
u(x) > u_\tau(x) \ge 1 - \gamma
$$
at points where $w_\tau(x) > 0$.
Using the monotonicity assumption \eqref{D53} on $g$, we obtain
\begin{align}\label{5D2}
g(x,u(x)) \le g(x,u_\tau(x)),
\qquad x' \in \mathbb{R}^{n-1}, \ |x_n| \le \lambda,
\end{align}
at points satisfying $w_\tau(x) > 0$.

\medskip
\noindent\textit{Case~2.} $x_n > \lambda$.

If $x_n > \lambda$, then by \eqref{5D1}, for any $x' \in \mathbb{R}^{n-1}$,
$$
u(x) > u_\tau(x) \ge 1 - \gamma
$$
at points where $w_\tau(x) > 0$.
Again, by the monotonicity assumption \eqref{D53} on $g$, it follows that
\begin{align}\label{5D3}
g(x,u(x)) \le g(x,u_\tau(x)),
\qquad x' \in \mathbb{R}^{n-1}, \ x_n > \lambda,
\end{align}
at points satisfying $w_\tau(x) > 0$.

\textit{Case 3.} $x_n < -\lambda$.

If $x_n < -\lambda$, then by \eqref{5D1}, for any $x' \in \mathbb{R}^{n-1}$, we have
$$
u_\tau(x) < u(x) \leq -1 + \gamma,
$$
at points where $w_\tau(x) > 0$. Using the monotonicity assumption \eqref{D53} on $g$, it follows that
\begin{align}\label{5D4}
g(x,u(x)) \le g(x,u_\tau(x)), \quad x' \in \mathbb{R}^{n-1},~ x_n < -\lambda,
\end{align}
for points satisfying $w_\tau(x) > 0$.

Combining \eqref{5D6}, \eqref{5D2}, \eqref{5D3}, and \eqref{5D4}, we conclude that if $w_\tau(x) > 0$, then
$$
(-\Delta)^s \tilde{W}_\tau(x) - \Delta \tilde{W}_\tau(x) \le g(x,u(x)) - g(x,u_\tau(x)) \le 0.
$$
More precisely, at points where $\tilde{W}_\tau(x) > 0$, we have
$$
(-\Delta)^s \tilde{W}_\tau(x) - \Delta \tilde{W}_\tau(x) \le g(x,u(x)) - g(x,u_\tau(x)) \le 0.
$$

Together with the exterior condition for $\tilde{W}_\tau(x)$ in $\mathcal{H}^c$ (see \eqref{5D5}) and the maximum principle in unbounded domains (Theorem \ref{D3}), this implies \eqref{q1}, i.e.,
$$
w_\tau(x) \le \frac{M}{2}, \quad x \in \mathbb{R}^n,
$$
which contradicts \eqref{D502}, where $\mathop{\sup}\limits_{\mathbb{R}^n} w_\tau(x) := M > 0$. Therefore, we conclude
\begin{align}\label{D508}
w_\tau(x) \le 0, \quad x \in \mathbb{R}^n, \quad \text{for any }~ \tau \ge 2\lambda.
\end{align}

Hence, \eqref{D501} holds.

\textbf{Step 2.} Inequality \eqref{D508} provides a starting point from which we perform the sliding method. In this step, starting from $\tau = 2\lambda$, we decrease $\tau$ and aim to show that for any $0 < \tau < 2\lambda$,
\begin{align}\label{D509}
w_\tau(x) \le 0, \quad x \in \mathbb{R}^n.
\end{align}

To prove \eqref{D509}, define
$$
\tau_0 := \inf\{\tau \mid w_\tau(x) \le 0,~ x \in \mathbb{R}^n\},
$$
and we will show that $\tau_0 = 0$. Otherwise, we can decrease $\tau_0$ slightly and still have
$$
w_\tau(x) \le 0, \quad x \in \mathbb{R}^n, \quad \forall \tau \in (\tau_0 - \varepsilon, \tau_0],
$$
which contradicts the definition of $\tau_0$.

First, we prove that
\begin{align}\label{D510}
\sup_{|x_n| \le \lambda,~ x' \in \mathbb{R}^{n-1}} w_{\tau_0}(x) < 0.
\end{align}
If \eqref{D510} does not hold, then
$$
\sup_{|x_n| \le \lambda,~ x' \in \mathbb{R}^{n-1}} w_{\tau_0}(x) = 0,
$$
so there exists a sequence
$$
\{x^k\} \subset \mathbb{R}^{n-1} \times [-\lambda, \lambda], \quad k = 1, 2, \dots
$$
such that
\begin{align}\label{q2}
w_{\tau_0}(x^k) \to 0 \quad \text{as }~ k \to +\infty.
\end{align}

Let
$$
\phi_k(x) := \phi(x - x^k), \quad \text{where } \quad
\phi(x) :=
\begin{cases}
a \, e^{\frac{1}{|x|^2 - 1}}, ~& |x| < 1, \\
0, ~& |x| \ge 1,
\end{cases}
$$
and take $a = e$ so that
\begin{align}\label{D511}
\phi(0) = \max_{\mathbb{R}^n} \phi(x) = 1.
\end{align}

From \eqref{q2}, there exists a sequence $\{\varepsilon_k\} \searrow 0$ such that
$$
w_{\tau_0}(x^k) + \varepsilon_k \phi_k(x^k) > 0.
$$

For any $x \in \mathbb{R}^n \setminus B_1(x^k)$, since $w_{\tau_0}(x) \le 0$ and $\phi_k(x) = 0$, we have
$$
w_{\tau_0}(x^k) + \varepsilon_k \phi_k(x^k) > w_{\tau_0}(x) + \varepsilon_k \phi_k(x), \quad \forall ~x \in \mathbb{R}^n \setminus B_1(x^k).
$$
Then there exists a point $\tilde{x}^k \in B_1(x^k)$ such that
\begin{align}\label{D512}
\varepsilon_k \ge w_{\tau_0}(\tilde{x}^k) + \varepsilon_k \phi_k(\tilde{x}^k)
= \max_{\mathbb{R}^n} \big( w_{\tau_0}(x) + \varepsilon_k \phi_k(x) \big) > 0,
\end{align}
which implies
\begin{align}\label{D513}
(-\Delta)^s (w_{\tau_0} + \varepsilon_k \phi_k)(\tilde{x}^k) - \Delta (w_{\tau_0} + \varepsilon_k \phi_k)(\tilde{x}^k)
&\ge C_{n, s}  \text{P.V.} \int_{\mathbb{R}^n} \frac{(w_{\tau_0} + \varepsilon_k \phi_k)(\tilde{x}^k) - (w_{\tau_0} + \varepsilon_k \phi_k)(y)}{|\tilde{x}^k - y|^{n+2s}} \, dy \nonumber\\
&\ge 0.
\end{align}

Meanwhile, by \eqref{D511} and \eqref{D512}, we have
$$
w_{\tau_0}(\tilde{x}^k) + \varepsilon_k \phi_k(\tilde{x}^k) \ge w_{\tau_0}(x^k) + \varepsilon_k \phi_k(x^k)
\quad \text{and} \quad
\phi_k(\tilde{x}^k) \le \phi_k(x^k),
$$
which implies
$$
0 \ge w_{\tau_0}(\tilde{x}^k) \ge w_{\tau_0}(x^k).
$$
Therefore, by \eqref{q2}, we obtain
\begin{equation}\label{q4}
w_{\tau_0}(\tilde{x}^k) \to 0 \quad \text{as }~ k \to +\infty.
\end{equation}

Combining \eqref{D51}, \eqref{D53}, and \eqref{D513}, we deduce
\begin{align}\label{D514}
0\leq &(-\Delta)^s (w_{\tau_0}+ \varepsilon_k \phi_k)(\tilde{x}^k)\nonumber\\
\leq&(-\Delta)^s (w_{\tau_0}+ \varepsilon_k \phi_k)(\tilde{x}^k)-\Delta (w_{\tau_0}+ \varepsilon_k \phi_k)(\tilde{x}^k)\nonumber\\
=&(-\Delta)^s w_{\tau_0}(\tilde{x}^k)+\varepsilon_k(-\Delta)^s  \phi_k(\tilde{x}^k)-\Delta w_{\tau_0}(\tilde{x}^k)-\varepsilon_k \Delta \phi_k(\tilde{x}^k)\nonumber\\
=&g(\tilde{x}^k,u(\tilde{x}^k))-g((\tilde{x}^k)',\tilde{x}^k_n+\tau_0,u_{\tau_0}(\tilde{x}^k))+\varepsilon_k(-\Delta)^s  \phi_k(\tilde{x}^k)-\varepsilon_k \Delta \phi_k(\tilde{x}^k)\nonumber\\
\leq&g(\tilde{x}^k,u(\tilde{x}^k))-g(\tilde{x}^k,u_{\tau_0}(\tilde{x}^k))+\varepsilon_k(-\Delta)^s  \phi_k(\tilde{x}^k)-\varepsilon_k \Delta \phi_k(\tilde{x}^k)\nonumber\\
\leq&g(\tilde{x}^k,u(\tilde{x}^k))-g(\tilde{x}^k,u_{\tau_0}(\tilde{x}^k))+\varepsilon_k C\to 0,~\mbox{ as }~k\to +\infty,
\end{align}
where $C$ is a positive constant. Here, we have used \eqref{q4} and the fact that
$$
|(-\Delta)^s \phi_k(\tilde{x}^k) - \Delta \phi_k(\tilde{x}^k)| \le C.
$$

Moreover, by Lemma \ref{L2} and \eqref{D512}, we have
\begin{align*}
&\frac{C_0}{C_{n,s}} r^{2s} \big\{(-\Delta)^s (w_{\tau_0} + \varepsilon_k \phi_k)(\tilde{x}^k) - \Delta (w_{\tau_0} + \varepsilon_k \phi_k)(\tilde{x}^k)\big\} \\
&\quad + C_0 r^{2s} \int_{B_r^c(\tilde{x}^k)} \frac{(w_{\tau_0} + \varepsilon_k \phi_k)(y)}{|\tilde{x}^k - y|^{n+2s}} \, dy \\
&\ge (w_{\tau_0} + \varepsilon_k \phi_k)(\tilde{x}^k), \quad \text{for any } r > 0,
\end{align*}
where
$$
C_0 = \frac{1}{\int_{B_r^c(0)} |y|^{-n-2s} \, dy}.
$$

Using \eqref{D512}, \eqref{q4}, \eqref{D514}, the above average inequality, and the fact that $\varepsilon_k \to 0$ as $k \to +\infty$, we obtain
$$
\int_{B_r^c(\tilde{x}^k)} \frac{w_{\tau_0}(y)}{|\tilde{x}^k - y|^{n+2s}} \, dy \to 0, \quad \text{as }~ k \to +\infty,
$$
which implies that for any fixed $r > 0$,
\begin{align}\label{D515}
w_{\tau_0}(x) \to 0, \quad x \in B_r^c(\tilde{x}^k), \quad \text{as }~ k \to +\infty.
\end{align}

Denote
$$
u_k(x) := u(x + \tilde{x}^k).
$$
Since $u(x)$ is uniformly continuous, by the Arzel\`{a}-Ascoli theorem, up to extracting a subsequence (still denoted by $u_k$), we have
$$
u_k(x) \to u_\infty(x), \quad x \in \mathbb{R}^n, \quad \text{as }~ k \to +\infty.
$$
Together with \eqref{D515}, we deduce
$$u_\infty(x)\equiv (u_\infty)_{\tau_0}(x),~x\in B_r^c(\tilde{x}^k).$$
Hence, for any $i \in \mathbb{N}$ and any fixed $r > 0$, we infer
\begin{align}\label{D516}
u_\infty(x', x_n) &= u_\infty(x', x_n + \tau_0) = u_\infty(x', x_n + 2\tau_0) \nonumber \\
&= \cdots = u_\infty(x', x_n + i\tau_0), \quad x \in B_r^c(\tilde{x}^k).
\end{align}

By the boundedness of the $n$-th component $\{\tilde{x}_n^k\}$ and the asymptotic condition \eqref{D52}, we have
$$
u_\infty(x', x_n) \longrightarrow \pm 1 \quad \text{as } x_n \to \pm \infty, \text{ uniformly in } x' = (x_1, \dots, x_{n-1}).
$$

Now, take $x_n$ sufficiently negative so that $u_\infty(x', x_n)$ is close to $-1$, and choose $i$ sufficiently large so that $u_\infty(x', x_n + i\tau_0)$ is close to $1$. This contradicts \eqref{D516}. Thus, \eqref{D510} is valid.

To show that there exists a small positive constant $\varepsilon$ such that
$$
w_\tau(x) \le 0, \quad x \in \mathbb{R}^n, \quad \forall \tau \in (\tau_0 - \varepsilon, \tau_0], \quad \text{for } \tau_0 > 0,
$$
note that \eqref{D510} implies the existence of such an $\varepsilon > 0$ satisfying
\begin{align}\label{D517}
\sup_{|x_n| \le \lambda,~ x' \in \mathbb{R}^{n-1}} w_\tau(x) \le 0, \quad \forall \tau \in (\tau_0 - \varepsilon, \tau_0].
\end{align}

It remains to show that
\begin{align}\label{D518}
\sup_{|x_n| > \lambda,~ x' \in \mathbb{R}^{n-1}} w_\tau(x) \le 0, \quad \forall \tau \in (\tau_0 - \varepsilon, \tau_0].
\end{align}

If \eqref{D518} does not hold, then there exist some $\tau \in (\tau_0 - \varepsilon, \tau_0]$ and a constant $M > 0$ such that
\begin{align}\label{D519}
\sup_{|x_n| > \lambda,~ x' \in \mathbb{R}^{n-1}} w_\tau(x) := M > 0.
\end{align}

By the asymptotic condition on $u$ in \eqref{D52}, there exists a constant $A > \lambda$ such that
\begin{align}\label{D520}
w_\tau(x) \le \frac{M}{2}, \quad x' \in \mathbb{R}^{n-1},~ |x_n| \ge A.
\end{align}

Denote
$$
\mathcal{M} := \{(x', x_n) \in \mathbb{R}^{n-1} \times \mathbb{R} \mid x' \in \mathbb{R}^{n-1},~ \lambda < |x_n| < A\},
$$
(see the blue section of Figure 2 below) and consider the differential inequality in $\mathcal{M}$ satisfied by
$$
\bar{W}_\tau(x) := w_\tau(x) - \frac{M}{2}.
$$

Combining \eqref{D517} and \eqref{D520}, we have the exterior condition
\begin{align}\label{D521}
\bar{W}_\tau(x) \le 0, \quad x \in \mathcal{M}^c, \quad \forall \tau \in (\tau_0 - \varepsilon, \tau_0].
\end{align}

\begin{center}
\begin{tikzpicture}[scale=1]
    \draw[->][thick,line width=1.5pt] (-5,0) -- (5,0) node at (5.3,0) {$x_n$};
    \draw[->][thick,line width=1.5pt] (0,-2.4) -- (0,2.7) node at (0,2.9) {$x'$};
    \draw[red,line width=2pt, dashed] (1,-2) -- (1,2);
    \draw[red,line width=2pt, dashed] (4,-2) -- (4,2);
    \draw[red,line width=2pt, dashed] (-1,-2) -- (-1,2);
    \draw[red,line width=2pt, dashed] (-4,-2) -- (-4,2);
    \fill[blue!60] (1,-2)--(1,2)--(4,2)--(4,-2);
    \fill[blue!60] (-1,-2)--(-1,2)--(-4,2)--(-4,-2);
    \node at (1,-2.4) {$x_n=\lambda$};
    \node at (4,-2.4) {$x_n=A$};
    \node at (-1,-2.4) {$x_n=-\lambda$};
    \node at (-4,-2.4) {$x_n=-A$};
\node [below=3.8cm, align=flush center,text width=12cm] at (0,0.5)
        {$Figure$ 2. The domain   $\mathcal{M}$.};
\end{tikzpicture}
\end{center}

For any $\tau \in (\tau_0 - \varepsilon, \tau_0]$ and $x' \in \mathbb{R}^{n-1}$, if $\lambda < x_n < A$ at points in $\mathcal{M}$ where $\bar{W}_\tau(x) > 0$, we have
$$
u(x) > u_\tau(x) + \frac{M}{2} \ge u_\tau(x) \ge 1 - \gamma.
$$
Then, by the monotonicity assumption \eqref{D53} on the function $g$, it follows that
\begin{align}\label{D522}
g(x,u(x)) \le g(x,u_\tau(x)), \quad x' \in \mathbb{R}^{n-1},~ \lambda < x_n < A,
\end{align}
at the points satisfying $\bar{W}_\tau(x) > 0$.

Similarly, if $-A < x_n < -\lambda$ at points in $\mathcal{M}$ where $\bar{W}_\tau(x) > 0$, for any $\tau \in (\tau_0 - \varepsilon, \tau_0]$ and $x' \in \mathbb{R}^{n-1}$, we have
$$
-1 + \gamma \ge u(x) > u_\tau(x).
$$
By the monotonicity assumption \eqref{D53} on $g$, we then obtain
\begin{align}\label{D523}
g(x,u(x)) \le g(x,u_\tau(x)), \quad x' \in \mathbb{R}^{n-1},~ -A < x_n < -\lambda,
\end{align}
at points satisfying $\bar{W}_\tau(x) > 0$.

Therefore, by \eqref{D51}, \eqref{D54}, \eqref{D522}, and \eqref{D523}, at points where $\bar{W}_\tau(x) > 0$, we deduce
\begin{align}\label{D524}
(-\Delta)^s \bar{W}_\tau(x) - \Delta \bar{W}_\tau(x)
&= (-\Delta)^s w_\tau(x) - \Delta w_\tau(x) \nonumber \\
&= g(x,u(x)) - g(x', x_n + \tau, u_\tau(x)) \nonumber \\
&\le g(x,u(x)) - g(x,u_\tau(x)) \nonumber \\
&\le 0, \quad x \in \mathcal{M},~ \forall \tau \in (\tau_0 - \varepsilon, \tau_0].
\end{align}

Applying \eqref{D521} and \eqref{D524} in Theorem \ref{D3}, we infer
$$
\bar{W}_\tau(x) \le 0, \quad x \in \mathbb{R}^n,~ \forall \tau \in (\tau_0 - \varepsilon, \tau_0],
$$
which contradicts \eqref{D519}. This proves that \eqref{D518} is valid, and therefore
$$
w_\tau(x) \le 0, \quad x \in \mathbb{R}^n,~ \forall \tau \in (\tau_0 - \varepsilon, \tau_0],
$$
contradicting the definition of $\tau_0$.

Hence, we conclude that \eqref{D509} is established.

\textbf{Step 3.}
In this step, we show that $u(x)$ is strictly increasing with respect to $x_n$ and that $u(x)$ depends only on $x_n$.
From the previous two steps, we have
$$
w_\tau(x) \le 0, \quad x \in \mathbb{R}^n,~ \forall~ \tau > 0.
$$
It remains to prove
\begin{align}\label{D526}
w_\tau(x) < 0, \quad x \in \mathbb{R}^n,~ \forall~ \tau > 0.
\end{align}

Suppose \eqref{D526} does not hold. Then there exists a point $\hat{x} \in \mathbb{R}^n$ and $\tau_0 > 0$ such that
\begin{align}\label{D527}
w_{\tau_0}(\hat{x}) = 0,
\end{align}
and $\hat{x}$ is a maximum point of $w_{\tau_0}(x)$ in $\mathbb{R}^n$.
It follows from \eqref{D51}, \eqref{D54}, and \eqref{D527} that
\begin{align}\label{D528}
(-\Delta)^s w_{\tau_0}(\hat{x}) - \Delta w_{\tau_0}(\hat{x})
&= g(\hat{x},u(\hat{x})) - g((\hat{x})', \hat{x}_n + \tau_0, u_{\tau_0}(\hat{x})) \nonumber\\
&\le g(\hat{x}, u(\hat{x})) - g(\hat{x}, u_{\tau_0}(\hat{x})) \nonumber\\
&= 0.
\end{align}

Since $w_{\tau_0}(y) \not\equiv 0$ in $\mathbb{R}^n$, we also have
$$
(-\Delta)^s w_{\tau_0}(\hat{x}) - \Delta w_{\tau_0}(\hat{x}) \ge (-\Delta)^s w_{\tau_0}(\hat{x})
= C_{n,s} \, \text{P.V.} \int_{\mathbb{R}^n} \frac{-w_{\tau_0}(y)}{|\hat{x}-y|^{n+2s}} \, dy > 0,
$$
which contradicts \eqref{D528}. Hence, \eqref{D526} holds.

Next, we show that $u(x)$ depends only on $x_n$.
If we replace $u_\tau(x)$ by $u(x + \tau \kappa)$ with $\kappa = (\kappa_1, \dots, \kappa_n)$ and $\kappa_n > 0$, the arguments in \textbf{Step 1} and \textbf{Step 2} still apply. More precisely, for each $\kappa$ with $\kappa_n > 0$, one has
$$
u(x + \tau \kappa) > u(x), \quad x \in \mathbb{R}^n,~ \forall~ \tau > 0.
$$

Letting $\kappa_n \to 0$ and using the continuity of $u(x)$, we infer
$$
u(x + \tau \kappa) \ge u(x)
$$
for arbitrary $\kappa$ with $\kappa_n = 0$. Replacing $\kappa$ with $-\kappa$, we then deduce
$$
u(x + \tau \kappa) \le u(x),
$$
for arbitrary $\kappa$ with $\kappa_n = 0$, which implies
$$
u(x + \tau \kappa) = u(x).
$$
Thus, $u(x',x_n)$ is independent of $x'$ and we conclude
$$
u(x) = u(x_n).
$$

This completes the proof of Theorem \ref{D5}.
\end{proof}

\section{Monotonicity and one-dimensional symmetry  for parabolic  equations}

 In this section, we establish Theorems \ref{N2}, \ref{s1} and \ref{T3}.

\subsection{Monotonicity in $\Omega \times \mathbb R$}

\begin{proof}[Proof of Theorem \ref{N2}]
For $\tau \ge 0$, define
$$
u_\tau(x,t) := u(x',x_n+\tau,t).
$$
The function $u_\tau$ is defined on the set
$$
\Omega_\tau := \Omega - \tau e_n,
$$
which is obtained from $\Omega$ by sliding it downward by a distance $\tau$ along the $x_n$-direction, where $e_n=(0,\dots,0,1)$.
Let
$$
E_\tau := \Omega_\tau \cap \Omega,
\qquad
\bar{\tau} := \sup\{\tau>0 \mid E_\tau \neq \emptyset\},
$$
and define
$$
w_\tau(x,t) := u(x,t)-u_\tau(x,t), \qquad (x,t)\in E_\tau \times \mathbb{R}.
$$

Since $u_\tau$ satisfies the same equation \eqref{N21} in $\Omega_\tau$ as $u$ does in $\Omega$, it follows from \eqref{N21} that
\begin{align}\label{N201}
\partial_t^\alpha w_\tau(x,t) + (-\Delta)^s w_\tau(x,t) - \Delta w_\tau(x,t)
= c(x,t)\, w_\tau(x,t),
\quad (x,t)\in E_\tau \times \mathbb{R},
\end{align}
where
$$
c(x,t) := \frac{h(t,u(x,t))-h(t,u_\tau(x,t))}{u(x,t)-u_\tau(x,t)}
$$
is an $L^\infty$ function satisfying $c(x,t)\le C$.
Moreover, by assumption $(\textbf{B})$, we have
\begin{equation}\label{q6}
w_\tau(x,t) \le 0, \quad (x,t)\in E_\tau^c \times \mathbb{R}.
\end{equation}

To prove the theorem, it suffices to show that
\begin{align}\label{N202}
w_\tau(x,t) < 0, \quad (x,t)\in E_\tau \times \mathbb{R},
\quad \text{for any }~ 0<\tau<\bar{\tau},
\end{align}
which is precisely the statement that $u$ is strictly increasing in the $x_n$–direction.

We now prove \eqref{N202} in two steps.

\textbf{Step 1.}
For $\tau$ sufficiently close to $\bar{\tau}$, the set $E_\tau$ is a narrow region. We claim that
\begin{align}\label{N203}
w_\tau(x,t) \le 0, \quad (x,t)\in E_\tau \times \mathbb{R}.
\end{align}
Indeed, combining \eqref{N201} and \eqref{q6}, and applying the narrow region principle (Theorem~\ref{N1}), we obtain \eqref{N203}.

\textbf{Step 2.}
The inequality \eqref{N203} provides a starting position from which the sliding procedure can be carried out. We then decrease $\tau$ continuously, as long as \eqref{N203} remains valid, until reaching its limiting position. Let
\begin{align*}
\tau_0 = \inf\left\{ \tau \mid w_\tau(x,t) \leq 0, ~ (x,t) \in E_\tau \times \mathbb{R},~ 0 < \tau < \bar{\tau} \right\}.
\end{align*}

We claim that
$$
\tau_0 = 0.
$$
Suppose, to the contrary, that $\tau_0 > 0$. We shall show that the domain $\Omega_\tau$ can be slid slightly upward and that
\begin{align}\label{N204}
w_\tau(x,t) \leq 0, \qquad (x,t) \in E_\tau \times \mathbb{R},
\quad \text{for any }~ \tau_0-\varepsilon < \tau \leq \tau_0,
\end{align}
which contradicts the definition of $\tau_0$.

Since
$$
w_{\tau_0}(x,t) \leq 0, \qquad (x,t) \in E_{\tau_0} \times \mathbb{R},
$$
and
$$
w_{\tau_0}(x,t) < 0, \qquad (x,t) \in (\Omega \cap \partial E_{\tau_0}) \times \mathbb{R},
$$
it follows from condition $(\textbf{B})$ that
$$
w_{\tau_0}(x,t) \not\equiv 0, \qquad (x,t) \in E_{\tau_0} \times \mathbb{R}.
$$

If there exists a point $(x^0,t_0)$ such that $w_{\tau_0}(x^0,t_0)=0$, then $(x^0,t_0)$ is a maximum point of $w_{\tau_0}$. Consequently,
$$
\left\{
\begin{array}{ll}
\partial_t w_{\tau_0}(x^0,t_0) \ge 0, \\
\partial_t^\alpha w_{\tau_0}(x^0,t_0)
= C_\alpha \displaystyle\int_{-\infty}^{t_0}
\frac{w_{\tau_0}(x^0,t_0)-w_{\tau_0}(x^0,\tau)}
{(t_0-\tau)^{1+\alpha}}\, d\tau
\ge 0, \quad 0<\alpha<1,
\end{array}
\right.
$$
and
$$
-\Delta w_{\tau_0}(x^0,t_0) \ge 0.
$$

By a direct computation, we obtain
\begin{align*}
&\partial_t^\alpha w_{\tau_0}(x^0,t_0)
+(-\Delta)^s w_{\tau_0}(x^0,t_0)
-\Delta w_{\tau_0}(x^0,t_0) \\
\ge{}& (-\Delta)^s w_{\tau_0}(x^0,t_0) \\
={}& C_{n,s}\,\mathrm{P.V.}\!\int_{\mathbb{R}^n}
\frac{w_{\tau_0}(x^0,t_0)-w_{\tau_0}(y,t_0)}
{|x^0-y|^{n+2s}}\,dy \\
>{}& 0,
\end{align*}
which contradicts
\begin{align*}
\partial_t^\alpha w_{\tau_0}(x^0,t_0)
+(-\Delta)^s w_{\tau_0}(x^0,t_0)
-\Delta w_{\tau_0}(x^0,t_0)
= h(t_0,u(x^0,t_0)) - h(t_0,u_{\tau_0}(x^0,t_0))
= 0,
\end{align*}
by the first inequality in \eqref{N21}. Therefore,
\begin{align}\label{N205}
w_{\tau_0}(x,t) < 0, \qquad (x,t) \in E_{\tau_0} \times \mathbb{R}.
\end{align}

Next, we choose a closed set $K \subset E_{\tau_0}$ such that
$E_{\tau_0}\setminus K$ is a narrow region. By \eqref{N205}, there exists a constant
$\tilde C_0>0$ satisfying
$$
w_{\tau_0}(x,t) \le -\tilde C_0 < 0,
\qquad (x,t) \in K \times \mathbb{R}.
$$
By the continuity of $w_\tau$ with respect to $\tau$, for sufficiently small
$\varepsilon>0$,
$$
w_{\tau_0-\varepsilon}(x,t) \le 0,
\qquad (x,t) \in K \times \mathbb{R}.
$$
It follows from condition $(\textbf{B})$ that
$$
w_{\tau_0-\varepsilon}(x,t) \le 0,
\qquad (x,t) \in (E_{\tau_0-\varepsilon})^c \times \mathbb{R}.
$$
Since
$$
(E_{\tau_0-\varepsilon}\setminus K)^c
= K \cup (E_{\tau_0-\varepsilon})^c,
$$
we conclude that
\begin{equation*}
\left\{
\begin{array}{ll}
\begin{aligned}
\partial^\alpha_t w_{\tau_0 - \varepsilon}(x,t)+(-\Delta)^s w_{\tau_0 - \varepsilon}(x,t)&-\Delta w_{\tau_0 - \varepsilon}(x,t)\\
&=c_{\tau_0 - \varepsilon}(x,t)w_{\tau_0 - \varepsilon}(x,t),
\end{aligned}
~& (x,t)\in (E_{\tau_0 - \varepsilon} \setminus K) \times \mathbb{R},\\
w_{\tau_0 - \varepsilon} \leq 0,~&(x,t)\in (E_{\tau_0 - \varepsilon} \setminus K)^c \times \mathbb{R}.\\
 \end{array}\right.
\end{equation*}
By Theorem~\ref{N1}, it follows that \eqref{N204} holds. Consequently, this contradicts the definition of $\tau_0$. Therefore,
$$
w_\tau(x,t) \le 0, \qquad (x,t)\in E_\tau \times \mathbb{R},
\quad \text{for any }~ 0<\tau<\bar{\tau}.
$$

Next, we prove \eqref{N202}. Suppose that there exists a point
$(\bar x,\bar t)$ such that $w_\tau(\bar x,\bar t)=0$. Then
$(\bar x,\bar t)$ is a maximum point of $w_\tau$, and hence
$$
\left\{
\begin{array}{ll}
\partial_t w_\tau(\bar x,\bar t)\ge 0,\\
\partial_t^\alpha w_\tau(\bar x,\bar t)
= C_\alpha \displaystyle\int_{-\infty}^{\bar t}
\frac{w_\tau(\bar x,\bar t)-w_\tau(\bar x,\tau)}
{(\bar t-\tau)^{1+\alpha}}\,d\tau
\ge 0,
\quad 0<\alpha<1,
\end{array}
\right.
$$
and
$$
-\Delta w_\tau(\bar x,\bar t)\ge 0.
$$

Since
$$
w_\tau(x,t)\not\equiv 0,
\qquad (x,t)\in E_\tau\times\mathbb{R},
\quad \text{for any } 0<\tau<\bar{\tau},
$$
we obtain
\begin{align*}
\partial_t^\alpha w_\tau(\bar x,\bar t)
+(-\Delta)^s w_\tau(\bar x,\bar t)
-\Delta w_\tau(\bar x,\bar t)
&\ge (-\Delta)^s w_\tau(\bar x,\bar t)\\
&= C_{n,s}\,\mathrm{P.V.}\!\int_{\mathbb{R}^n}
\frac{w_\tau(\bar x,\bar t)-w_\tau(y,\bar t)}
{|\bar x-y|^{n+2s}}\,dy\\
&> 0,
\end{align*}
which contradicts
$$
\partial_t^\alpha w_\tau(\bar x,\bar t)
+(-\Delta)^s w_\tau(\bar x,\bar t)
-\Delta w_\tau(\bar x,\bar t)
= h(t,u(\bar x,\bar t))
- h(t,u_\tau(\bar x,\bar t))
= 0,
$$
by the first inequality in \eqref{N21}. Hence, \eqref{N202} follows and the proof
of Theorem~\ref{N2} is completed.
\end{proof}

\subsection{Monotonicity in $\mathbb{R}_+^n$}

\begin{proof}[Proof of Theorem \ref{s1}]
First, define
$$
u_\tau(x,t)=u(x',x_n+\tau,t)
\quad \text{and} \quad
w_\tau(x,t)=u(x,t)-u_\tau(x,t),
$$
for any $\tau>0$. Since $h(t,u)$ is monotone decreasing in $u$, we have
\begin{align}\label{s101}
h(t,u(x,t)) - h(t,u_\tau(x,t)) \le 0,
\qquad (x,t)\in \mathbb{R}_+^n \times \mathbb{R},
\end{align}
at points where $w_\tau(x,t)>0$.

Combining \eqref{s11} and \eqref{s101}, we deduce that
\begin{align*}
\partial_t^\alpha w_\tau(x,t)
+(-\Delta)^s w_\tau(x,t)
-\Delta w_\tau(x,t)
= h(t,u(x,t)) - h(t,u_\tau(x,t))
\le 0,
\end{align*}
for $(x,t)\in \mathbb{R}_+^n \times \mathbb{R}$ at points where $w_\tau(x,t)>0$.
It is easy to verify that the exterior condition
$$
w_\tau(x,t)\le 0,
\qquad (x,t)\in (\mathbb{R}^n\setminus \mathbb{R}_+^n)\times \mathbb{R},
$$
is satisfied. Applying the maximum principle in unbounded domains
(Theorem~\ref{T2}), we obtain
\begin{align}\label{s102}
w_\tau(x,t)\le 0,
\qquad (x,t)\in \mathbb{R}_+^n \times \mathbb{R},
\quad \text{for all }~ \tau>0.
\end{align}

Finally, we show that
\begin{align}\label{s103}
w_\tau(x,t)<0,
\qquad (x,t)\in \mathbb{R}_+^n \times \mathbb{R},
\quad \text{for all }~ \tau>0.
\end{align}
Suppose, by contradiction, that there exist a point $(x^0,t_0)$ and
$\tau_0>0$ such that
$$
w_{\tau_0}(x^0,t_0)
:= \max_{\mathbb{R}_+^n \times \mathbb{R}} w_{\tau_0}(x,t)
= 0.
$$
Then $(x^0,t_0)$ is a maximum point of $w_{\tau_0}$, and hence
$$
\left\{
\begin{array}{ll}
\partial_t w_{\tau_0}(x^0,t_0)\ge 0,\\
\partial_t^\alpha w_{\tau_0}(x^0,t_0)
= C_\alpha \displaystyle\int_{-\infty}^{t_0}
\frac{w_{\tau_0}(x^0,t_0)-w_{\tau_0}(x^0,\tau)}
{(t_0-\tau)^{1+\alpha}}\,d\tau
\ge 0,
\quad 0<\alpha<1,
\end{array}
\right.
$$
and
$$
-\Delta w_{\tau_0}(x^0,t_0)\ge 0.
$$

Since $u=0$ in $\mathbb{R}^n\setminus \mathbb{R}_+^n$ and $u>0$ in
$\mathbb{R}_+^n$, it follows that
$w_{\tau_0}(y,t_0)\not\equiv 0$ for $y\in\mathbb{R}^n$.
Moreover, by \eqref{s11}, \eqref{s102}, and the definitions of the
fractional Laplacian and the Laplacian, we infer that
\begin{align*}
\partial_t^\alpha w_{\tau_0}(x^0,t_0)
+(-\Delta)^s w_{\tau_0}(x^0,t_0)
-\Delta w_{\tau_0}(x^0,t_0)
&\ge C_{n,s}\,\mathrm{P.V.}\!\int_{\mathbb{R}^n}
\frac{-w_{\tau_0}(y,t_0)}{|x^0-y|^{n+2s}}\,dy \\
&> 0,
\end{align*}
which contradicts
\begin{align*}
\partial_t^\alpha w_{\tau_0}(x^0,t_0)
+(-\Delta)^s w_{\tau_0}(x^0,t_0)
-\Delta w_{\tau_0}(x^0,t_0)
= h(t_0,u(x^0,t_0)) - h(t_0,u_{\tau_0}(x^0,t_0))
= 0.
\end{align*}

This contradiction proves \eqref{s103} and completes the proof of
Theorem~\ref{s1}.
\end{proof}

\subsection{One-dimensional symmetry in $\mathbb{R}^n$}

\begin{proof}[Proof of Theorem \ref{T3}]
Denote
$$
W_\tau(x,t)=u(x,t)-u_\tau(x,t).
$$

We prove Theorem~\ref{T3} in three steps.

\medskip
\noindent\textbf{Step~1.}
We show that for $\tau$ sufficiently large,
\begin{align}\label{T301}
W_\tau(x,t)\le 0,
\qquad (x,t)\in \mathbb{R}^n \times \mathbb{R}.
\end{align}
According to \eqref{T32}, there exists a sufficiently large constant
$\lambda>0$ such that
\begin{align}\label{T302}
|u(x,t)| \ge 1-\gamma,
\qquad (x',t)\in \mathbb{R}^{n-1}\times \mathbb{R},
\quad |x_n|\ge \lambda.
\end{align}

We argue by contradiction. If \eqref{T301} does not hold, then there exists
a constant $M>0$ such that
\begin{align}\label{T303}
\sup_{(x,t)\in \mathbb{R}^n \times \mathbb{R}} W_\tau(x,t)
=: M >0.
\end{align}
Define
\begin{align}\label{T304}
\widetilde W_\tau(x,t)
= W_\tau(x,t)-\frac{M}{2}.
\end{align}
We shall prove that
$$
\widetilde W_\tau(x,t)\le 0,
\qquad (x,t)\in \mathbb{R}^n \times \mathbb{R},
$$
by applying the maximum principle in unbounded domains
(Theorem~\ref{T2}).

For all $t\in\mathbb{R}$, by \eqref{T32}, we can choose a sufficiently large
constant $A>\lambda$ such that
\begin{align}\label{T306}
\widetilde W_\tau(x,t)\le 0,
\qquad (x',t)\in \mathbb{R}^{n-1}\times \mathbb{R},
\quad x_n\ge A.
\end{align}
Let
$$
\mathcal D=\mathbb{R}^{n-1}\times(-\infty,A).
$$
Then \eqref{T306} implies
\begin{align}\label{T307}
\widetilde W_\tau(x,t)\le 0,
\qquad (x,t)\in \mathcal D^c \times \mathbb{R},
\end{align}
and hence $\widetilde W_\tau$ satisfies the exterior condition of
Theorem~\ref{T2}.

Next, we verify the differential inequality satisfied by
$\widetilde W_\tau(x,t)$ in $\mathcal D\times \mathbb{R}$.
Combining \eqref{T31} with the definition of $\widetilde W_\tau$, we obtain
\begin{align}\label{T308}
\partial_t^\alpha \widetilde W_\tau(x,t)
+(-\Delta)^s \widetilde W_\tau(x,t)
-\Delta \widetilde W_\tau(x,t)
&= \partial_t^\alpha W_\tau(x,t)
+(-\Delta)^s W_\tau(x,t)
-\Delta W_\tau(x,t) \nonumber\\
&= h(t,u(x,t)) - h(t,u_\tau(x,t)).
\end{align}

Next, we distinguish three possible cases.

\textit{Case 1).} $|x_n| \leq \lambda$.

If $|x_n| \leq \lambda$, then for any $\tau \geq 2\lambda$ we have
$$
x_n + \tau \geq \lambda .
$$
By \eqref{T302}, for $(x', t) \in \mathbb{R}^{n-1} \times \mathbb{R}$,
$$
u(x,t) > u_\tau(x,t) \geq 1-\gamma
$$
at points where $W_\tau(x,t) > 0$.
By the monotonicity assumption \eqref{T33} on the function $h$, this yields
\begin{align}\label{T309}
h(t,u(x,t)) \leq h(t,u_\tau(x,t)),
\quad (x',t)\in\mathbb{R}^{n-1}\times\mathbb{R},\ |x_n|\leq\lambda,
\end{align}
at points where $W_\tau(x,t) > 0$.

\medskip
\textit{Case 2).} $x_n > \lambda$.

If $x_n > \lambda$, then by \eqref{T32} and \eqref{T302}, for
$(x', t) \in \mathbb{R}^{n-1} \times \mathbb{R}$,
$$
u(x,t) > u_\tau(x,t) \geq 1-\gamma
$$
at points where $W_\tau(x,t) > 0$.
Using again the monotonicity assumption \eqref{T33} on $h$, we obtain
\begin{align}\label{T310}
h(t,u(x,t)) \leq h(t,u_\tau(x,t)),
\quad (x',t)\in\mathbb{R}^{n-1}\times\mathbb{R},\ x_n>\lambda,
\end{align}
at points where $W_\tau(x,t) > 0$.

\medskip
\textit{Case 3).} $x_n < -\lambda$.

If $x_n < -\lambda$, then by \eqref{T32} and \eqref{T302}, for
$(x', t) \in \mathbb{R}^{n-1} \times \mathbb{R}$,
$$
u_\tau(x,t) < u(x,t) \leq -1+\gamma
$$
at points where $W_\tau(x,t) > 0$.
By the monotonicity assumption \eqref{T33} on $h$, we deduce
\begin{align}\label{T311}
h(t,u(x,t)) \leq h(t,u_\tau(x,t)),
\quad (x',t)\in\mathbb{R}^{n-1}\times\mathbb{R},\ x_n<-\lambda,
\end{align}
at points where $W_\tau(x,t) > 0$.

\medskip
Combining \eqref{T308}, \eqref{T309}, \eqref{T310}, and \eqref{T311}, we infer that
\begin{align*}
\partial_t^\alpha \tilde{W}_\tau(x,t)
+ (-\Delta)^s \tilde{W}_\tau(x,t)
- \Delta \tilde{W}_\tau(x,t)
= h(t,u(x,t)) - h(t,u_\tau(x,t)) \leq 0
\end{align*}
at points where $W_\tau(x,t) > 0$.

Moreover,
\begin{align*}
{\partial^\alpha_ t} \tilde{W}_\tau (x, t) + (-\Delta)^s \tilde{W}_\tau(x, t) -\Delta \tilde{W}_\tau(x, t)=h(t, u(x, t)) -h(t, u_\tau(x, t))\leq 0,
\end{align*}
at  points where $\tilde{W}_\tau(x, t) > 0.$ Together with the exterior condition on $\tilde{W}_\tau(x, t)$ in $\mathcal{D}^c \times \mathbb{R}$ given by \eqref{T307}, and the maximum principle in unbounded domains (Theorem \ref{T2}), this implies
\begin{align}\label{T312}
\tilde{W}_\tau(x,t) \leq 0,
\quad (x,t)\in\mathbb{R}^n \times \mathbb{R}.
\end{align}

Combining \eqref{T304} with \eqref{T312}, we obtain
$$
W_\tau(x,t) \leq \frac{M}{2},
\quad (x,t)\in\mathbb{R}^n \times \mathbb{R},
$$
which contradicts \eqref{T303}.
Therefore, for any $\tau \geq 2\lambda$,
\begin{align}\label{T313}
W_\tau(x,t) \leq 0,
\quad (x,t)\in\mathbb{R}^n \times \mathbb{R}.
\end{align}

This completes the proof of \textbf{Step~1}.

\textbf{Step 2.}
Since inequality \eqref{T313} obtained in \textbf{Step~1} provides a starting position,
we can now proceed with the sliding argument.
In this step, we decrease $\tau$ starting from $\tau = 2\lambda$ and show that
for any $0 < \tau < 2\lambda$ the inequality
\begin{align}\label{T3201}
W_\tau(x,t) \leq 0,
\quad (x,t)\in\mathbb{R}^n \times \mathbb{R}
\end{align}
still holds.

Define
\begin{align*}
\tau_0
:= \inf\Big\{ \tau \;\big|\;
W_\tau(x,t) \leq 0,\ (x,t)\in\mathbb{R}^n \times \mathbb{R} \Big\}.
\end{align*}
We aim to prove that $\tau_0 = 0$.
Otherwise, we will show that $\tau_0$ can be decreased slightly while the inequality
\begin{align*}
W_\tau(x,t) \leq 0,
\quad (x,t)\in\mathbb{R}^n \times \mathbb{R},
\quad \forall\, \tau \in (\tau_0-\varepsilon,\,\tau_0]
\end{align*}
still holds for some $\varepsilon>0$,
which contradicts the definition of $\tau_0$.

\medskip
First, we claim that
\begin{align}\label{T3209}
\sup_{\substack{|x_n|\leq \lambda \\ (x',t)\in\mathbb{R}^{n-1}\times\mathbb{R}}}
W_{\tau_0}(x,t) < 0.
\end{align}
This immediately implies that there exists a sufficiently small constant
$\varepsilon>0$ such that
\begin{align}\label{T3210}
\sup_{\substack{|x_n|\leq \lambda \\ (x',t)\in\mathbb{R}^{n-1}\times\mathbb{R}}}
W_\tau(x,t) \leq 0,
\quad \forall\, \tau \in (\tau_0-\varepsilon,\,\tau_0].
\end{align}

Suppose, by contradiction, that \eqref{T3209} does not hold.
Then
\begin{align*}
\sup_{\substack{|x_n|\leq \lambda \\ (x',t)\in\mathbb{R}^{n-1}\times\mathbb{R}}}
W_{\tau_0}(x,t) = 0,
\end{align*}
and there exists a sequence
\begin{align*}
\{(x^k,t_k)\}
\subset (\mathbb{R}^{n-1}\times[-\lambda,\lambda]) \times \mathbb{R},
\quad k=1,2,\dots,
\end{align*}
such that
\begin{align*}
W_{\tau_0}(x^k,t_k) \to 0
\quad \text{as }~ k \to \infty.
\end{align*}
Meanwhile, there exists a nonnegative sequence $\{\varepsilon_k\} \searrow 0$ such that
\begin{align}\label{T3211}
W_{\tau_0}(x^k, t_k) = -\varepsilon_k.
\end{align}

Now we introduce the following auxiliary function:
\begin{align}\label{T3212}
W(x, t) = W_{\tau_0}(x, t) + \varepsilon_k \, \xi_k(x, t),
\end{align}
where
$$
\xi_k(x, t) = \xi(x - x^k, t - t_k),
$$
with $\xi(x, t) \in C_0^\infty(\mathbb{R}^n \times \mathbb{R})$ defined by
\begin{equation*}
\xi(x, t)=
\begin{cases}
1, ~& |x|, |t| \leq \frac{1}{2}, \\
0, ~& |x|, |t| \geq 1.
\end{cases}
\end{equation*}

Let
$$
Q(x^k, t_k) := \{(x, t) \mid |x - x^k|, |t - t_k| < 1\}.
$$
Then the maximum of the perturbed function $W(x, t)$ over $Q(x^k, t_k)$ is at least as large as its value outside $Q(x^k, t_k)$. Hence, the global maximum of $W(x, t)$ in $\mathbb{R}^n \times \mathbb{R}$ is attained in $Q(x^k, t_k)$. In particular, combining \eqref{T3211} and \eqref{T3212}, we have $W(x^k, t_k) = 0$. For $(x, t) \in (\mathbb{R}^n \times \mathbb{R}) \setminus Q(x^k, t_k)$, we have $\xi_k(x, t) = 0$, which implies $W(x, t) \leq 0$. Therefore, there exists $(\bar{x}^k, \bar{t}_k) \in Q(x^k, t_k)$ such that $W(x, t)$ attains its maximum:
\begin{align}\label{T3213}
0 \leq W(\bar{x}^k, \bar{t}_k) := \sup_{\mathbb{R}^n \times \mathbb{R}} W(x, t) \leq \varepsilon_k.
\end{align}
From the definition \eqref{T3212} and the estimate \eqref{T3213}, we deduce
\begin{align}\label{T3217}
W_{\tau_0}(\bar{x}^k, \bar{t}_k) \to 0, \quad \text{as }~ k \to +\infty.
\end{align}

Moreover, by \eqref{T3213},
\begin{align}\label{T3214}
0 \leq W(\bar{x}^k, \bar{t}_k) := \sup_{\mathbb{R}^n \times \mathbb{R}} W(x, t) \leq \varepsilon_k.
\end{align}
Meanwhile, at $(\bar{x}^k, \bar{t}_k)$ we have
\begin{equation}\label{T3215}
\begin{cases}
\partial_t W(\bar{x}^k, \bar{t}_k) \geq 0,\\
\partial_t^\alpha W(\bar{x}^k, \bar{t}_k) = C_\alpha \int_{-\infty}^{\bar{t}_k} \frac{W(\bar{x}^k, \bar{t}_k) - W(\bar{x}^k, \tau)}{(\bar{t}_k - \tau)^{1+\alpha}} \, d\tau \geq 0, & \text{if }~ 0<\alpha<1,
\end{cases}
\end{equation}
and
$$
-\Delta W(\bar{x}^k, \bar{t}_k) \geq 0.
$$
Therefore,
\begin{align}\label{T3216}
(-\Delta)^s W(\bar{x}^k, \bar{t}_k) - \Delta W(\bar{x}^k, \bar{t}_k)
\geq (-\Delta)^s W(\bar{x}^k, \bar{t}_k)
= C_{n,s} \, \text{P.V.} \int_{\mathbb{R}^n} \frac{W(\bar{x}^k, \bar{t}_k) - W(y, \bar{t}_k)}{|\bar{x}^k - y|^{n+2s}} \, dy \geq 0.
\end{align}
Combining \eqref{T31}, \eqref{T3212}, \eqref{T3217}, \eqref{T3215}, and \eqref{T3216}, we obtain
\begin{align}\label{T3218}
0 \leq&(-\Delta)^sW(\bar x^k, \bar{t}_k)-\Delta W(\bar x^k, \bar{t}_k)\nonumber\\
=& (-\Delta)^s W_{\tau_0}(\bar{x}^k, \bar{t}_k) + \varepsilon_k (-\Delta)^s \xi_k(\bar{x}^k, \bar{t}_k)-\Delta W_{\tau_0}(\bar{x}^k, \bar{t}_k) - \varepsilon_k \Delta \xi_k(\bar{x}^k, \bar{t}_k) \nonumber\\
=&-{\partial^\alpha_ t} W_{\tau_0}(\bar{x}^k, \bar{t}_k)+h(\bar{t}_k, u(\bar{x}^k, \bar{t}_k))-h(\bar{t}_k, u_{\tau_0}(\bar{x}^k, \bar{t}_k))
+\varepsilon_k (-\Delta)^s \xi_k(\bar{x}^k, \bar{t}_k)- \varepsilon_k \Delta \xi_k(\bar{x}^k, \bar{t}_k) \nonumber\\
=&-\{{\partial^\alpha_ t} W(\bar{x}^k, \bar{t}_k)-\varepsilon_k{\partial^\alpha_ t} \xi_k (\bar{x}^k, \bar{t}_k)\}+h(\bar{t}_k, u(\bar{x}^k, \bar{t}_k))-h(\bar{t}_k, u_{\tau_0}(\bar{x}^k, \bar{t}_k))\nonumber\\
&+\varepsilon_k (-\Delta)^s \xi_k(\bar{x}^k, \bar{t}_k)- \varepsilon_k \Delta \xi_k(\bar{x}^k, \bar{t}_k) \nonumber\\
\leq&h(\bar{t}_k, u(\bar{x}^k, \bar{t}_k))-h(\bar{t}_k, u_{\tau_0}(\bar{x}^k, \bar{t}_k))
+\varepsilon_k (-\Delta)^s \xi_k(\bar{x}^k, \bar{t}_k)- \varepsilon_k \Delta \xi_k(\bar{x}^k, \bar{t}_k)+\varepsilon_k{\partial^\alpha_ t} \xi_k(\bar{x}^k, \bar{t}_k)\nonumber\\
\leq&h(\bar{t}_k, u(\bar{x}^k, \bar{t}_k))-h(\bar{t}_k, u_{\tau_0}(\bar{x}^k, \bar{t}_k))
+\varepsilon_k C\nonumber\\
\to& 0, ~\mbox{ as }~k\to +\infty,
\end{align}
where we used the fact that
$$
\big|(-\Delta)^s \xi_k(\bar{x}^k, \bar{t}_k) - \Delta \xi_k(\bar{x}^k, \bar{t}_k) + \partial_t^\alpha \xi_k(\bar{x}^k, \bar{t}_k)\big| \leq C,
$$
and $C$ is a positive constant. Moreover, by Lemma \ref{L1} and \eqref{T3214}, for any $r>0$, we obtain
\begin{align}\label{T3219}
\frac{C_0}{C_{n,s}} r^{2s} \big\{(-\Delta)^s W(\bar{x}^k, \bar{t}_k) - \Delta W(\bar{x}^k, \bar{t}_k)\big\}
+ C_0 r^{2s} \int_{B_r^c(\bar{x}^k)} \frac{W(y, \bar{t}_k)}{|\bar{x}^k - y|^{n+2s}} \, dy \geq W(\bar{x}^k, \bar{t}_k),
\end{align}
with
$$
C_0 = \frac{1}{\int_{B_1^c(0)} |y|^{-n-2s} \, dy}.
$$
By \eqref{T3214}, \eqref{T3218} and \eqref{T3219}, for any finite $r>0$, we have
$$
\int_{B_r^c(\bar{x}^k)} \frac{W(y, \bar{t}_k)}{|\bar{x}^k - y|^{n+2s}} \, dy \to 0, \quad \text{as }~ k \to +\infty.
$$
Furthermore, for any fixed $r>0$,
\begin{align}\label{T3220}
W(y, \bar{t}_k) \to 0, \quad y \in B_r^c(\bar{x}^k), \quad \text{as }~ k \to +\infty.
\end{align}

Denote
$$
u_k(x, t) := u(x + \bar{x}^k, t + \bar{t}_k).
$$
Since $u(x, t)$ is uniformly continuous, by the Arzelà-Ascoli theorem, there exists a subsequence (still denoted $u_k$) such that
$$
u_k(x, t) \to u_\infty(x, t), \quad (x, t) \in \mathbb{R}^n \times \mathbb{R}, \quad \text{as }~ k \to +\infty.
$$
Combining this with \eqref{T3220}, we have
$$
u_\infty(x, \bar{t}_k) - (u_\infty)_{\tau_0}(x, \bar{t}_k) \equiv 0, \quad x \in B_r^c(\bar{x}^k).
$$
Therefore, for any $i \in \mathbb{N}$ and any fixed $r>0$,
\begin{align}\label{T3221}
u_\infty(x', x_n, \bar{t}_k)
&= u_\infty(x', x_n + \tau_0, \bar{t}_k)
= u_\infty(x', x_n + 2\tau_0, \bar{t}_k)
= \cdots \nonumber\\
&= u_\infty(x', x_n + i\tau_0, \bar{t}_k), \quad x \in B_r^c(\bar{x}^k).
\end{align}
Due to the boundedness of the $n$-th component $\bar{x}_n^k$ of $\bar{x}^k$, and the asymptotic condition \eqref{T32}, we deduce
$$
u_\infty(x', x_n, \bar{t}_k) \longrightarrow \pm 1 \quad \text{as } x_n \to \pm \infty, \text{ uniformly in } x' = (x_1, \dots, x_{n-1}).
$$
Hence, by taking $x_n$ sufficiently negative, we can make $u_\infty(x', x_n, \bar{t}_k)$ close to $-1$, and then choose $i$ sufficiently large to make $u_\infty(x', x_n + i\tau_0, \bar{t}_k)$ close to $1$, which contradicts \eqref{T3221}. This proves \eqref{T3209} and therefore \eqref{T3210} holds.

Next, for $|x_n| > \lambda$, it suffices to show that
\begin{align}\label{T3203}
\sup_{\substack{|x_n|>\lambda \\ (x',t)\in\mathbb{R}^{n-1}\times\mathbb{R}}}
 W_\tau(x, t) \leq 0, \quad \forall~ \tau \in (\tau_0 - \varepsilon, \tau_0].
\end{align}
If \eqref{T3203} is false, then there exist some $\tau \in (\tau_0 - \varepsilon, \tau_0]$ and a constant $M > 0$ such that
\begin{align}\label{T3204}
\sup_{\substack{|x_n|>\lambda \\ (x',t)\in\mathbb{R}^{n-1}\times\mathbb{R}}} W_\tau(x, t) := M > 0.
\end{align}
By the asymptotic condition on $u$ in \eqref{T32}, there exists a constant $A > \lambda$ such that
\begin{align}\label{T3205}
W_\tau(x, t) \leq \frac{M}{2}, \quad (x', t) \in \mathbb{R}^{n-1} \times \mathbb{R}, \quad |x_n| \geq A.
\end{align}

Denote
$$
\mathcal{E} = \{(x', x_n, t) \in \mathbb{R}^{n-1} \times \mathbb{R} \times \mathbb{R} \mid (x', t) \in \mathbb{R}^{n-1} \times \mathbb{R},~ \lambda < |x_n| < A \},
$$
and consider the differential inequality satisfied by the function
$$
\check{W}_\tau(x, t) := W_\tau(x, t) - \frac{M}{2} \quad \text{in }~ \mathcal{E}.
$$

If $\lambda < x_n < A$, then at the points in $\mathcal{E}$ where $\check{W}_\tau(x, t) > 0$, for all $\tau \in (\tau_0 - \varepsilon, \tau_0]$ and $(x', t) \in \mathbb{R}^{n-1} \times \mathbb{R}$, we have
$$
u(x, t) > u_\tau(x, t) + \frac{M}{2} \geq u_\tau(x, t) \geq 1 - \gamma.
$$
Then, by the monotonicity assumption \eqref{T33} on the function $h$, we deduce
\begin{align}\label{T3206}
h(t, u(x, t)) \leq h(t, u_\tau(x, t)), \quad (x', t) \in \mathbb{R}^{n-1} \times \mathbb{R}, \quad \lambda < x_n < A,
\end{align}
at the points where $\check{W}_\tau(x, t) > 0$.

If $-A < x_n < -\lambda$ at the points in $\mathcal{E}$ where $\check{W}_\tau(x, t) > 0$, for all $\tau \in (\tau_0 - \varepsilon, \tau_0]$ and $(x', t) \in \mathbb{R}^{n-1} \times \mathbb{R}$, we have
$$
-1+\gamma \geq u(x, t) > u_\tau(x, t).
$$
Then, by the monotonicity assumption \eqref{T33} on the function $h$, we obtain
\begin{align}\label{T3207}
h(t, u(x, t)) \leq h(t, u_\tau(x, t)), \quad (x', t) \in \mathbb{R}^{n-1} \times \mathbb{R}, \quad -A < x_n < -\lambda,
\end{align}
at the points where $\check{W}_\tau(x, t) > 0$.

Combining \eqref{T31}, \eqref{T3206}, and \eqref{T3207}, we deduce
\begin{align}\label{T3208}
\partial^\alpha_t \check{W}_\tau(x, t) + (-\Delta)^s \check{W}_\tau(x, t) - \Delta \check{W}_\tau(x, t)
&= \partial^\alpha_t W_\tau(x, t) + (-\Delta)^s W_\tau(x, t) - \Delta W_\tau(x, t) \nonumber\\
&= h(t, u(x, t)) - h(t, u_\tau(x, t)) \nonumber\\
&\leq 0,
\end{align}
at the points in $\mathcal{E}$ where $\check{W}_\tau(x, t) > 0$.

By combining \eqref{T3210} and \eqref{T3205}, we deduce the following exterior condition:
$$
\check{W}_\tau(x, t) \leq 0, \quad (x, t) \in \mathcal{E}^c \times \mathbb{R}.
$$
Then, applying \eqref{T3208} and the maximum principle in unbounded domains (Theorem \ref{T2}), we conclude
$$
\check{W}_\tau(x, t) \leq 0, \quad (x, t) \in \mathbb{R}^n \times \mathbb{R}, \quad \forall~ \tau \in (\tau_0 - \varepsilon, \tau_0],
$$
which contradicts \eqref{T3204}. Hence, \eqref{T3203} holds.

It follows from \eqref{T3210} and \eqref{T3203} that \eqref{T3209} is true, i.e.,
$$
W_\tau(x, t) \leq 0, \quad (x, t) \in \mathbb{R}^n \times \mathbb{R}, \quad \forall~ \tau \in (\tau_0 - \varepsilon, \tau_0],
$$
which contradicts the definition of $\tau_0$.

Therefore, \eqref{T3201} holds, and this completes the proof in \textbf{Step 2}.

\textbf{Step 3.}
First, we show that
$$
u(x, t) \text{ is strictly increasing with respect to } x_n
$$
and
$$
u(x, t) = u(x_n, t).
$$

By \textbf{Step 1} and \textbf{Step 2}, we have
$$
W_\tau(x, t) \leq 0, \quad (x, t) \in \mathbb{R}^{n} \times \mathbb{R}, \quad \forall~ \tau > 0.
$$
We only need to prove that
\begin{align}\label{T3301}
W_\tau(x, t) < 0, \quad (x, t) \in \mathbb{R}^{n} \times \mathbb{R}, \quad \forall~ \tau > 0.
\end{align}

Suppose \eqref{T3301} does not hold. Then there exists a point $(x^0, t_0) \in \mathbb{R}^{n} \times \mathbb{R}$ and some $\tau_0 > 0$ such that
\begin{align}\label{T3302}
W_{\tau_0}(x^0, t_0) = 0,
\end{align}
and $(x^0, t_0)$ is a maximum point of $W_{\tau_0}(x, t)$ in $\mathbb{R}^n \times \mathbb{R}$, i.e.,
$$
W_{\tau_0}(x^0, t_0) := \sup_{\mathbb{R}^{n} \times \mathbb{R}} W_{\tau_0}(x, t) = 0.
$$
It follows that
\begin{equation}\label{T3303}
\left\{
\begin{array}{ll}
{\partial_t} W_{\tau_0}(x^0, t_0) \geq 0,\\
{\partial^\alpha_t} W_{\tau_0}(x^0, t_0) = C_\alpha \int_{-\infty}^{t_0} \frac{W_{\tau_0}(x^0, t_0) - W_{\tau_0}(x^0, \tau)}{(t_0 - \tau)^{1+\alpha}} \, d\tau \geq 0, \quad \text{if }~ 0 < \alpha < 1,
\end{array}
\right.
\end{equation}
and
\begin{align}\label{T3304}
-\Delta W_{\tau_0}(x^0, t_0) \geq 0.
\end{align}

Combining \eqref{T31}, \eqref{T3302}, \eqref{T3303}, and \eqref{T3304}, and using the definitions of the fractional Laplacian and Laplacian, we deduce
\begin{align*}
0 &= h(t_0, u(x^0, t_0)) - h(t_0, u_{\tau_0}(x^0, t_0)) \\
&= {\partial^\alpha_t} W_{\tau_0}(x^0, t_0) + (-\Delta)^s W_{\tau_0}(x^0, t_0) - \Delta W_{\tau_0}(x^0, t_0)\\
&= {\partial^\alpha_t} W_{\tau_0}(x^0, t_0) + C_{n,s} \, P.V. \int_{\mathbb{R}^n} \frac{-W_{\tau_0}(y, t_0)}{|x^0 - y|^{n+2s}} \, dy - \Delta W_{\tau_0}(x^0, t_0),
\end{align*}
that is,
$$
C_{n,s} \, P.V. \int_{\mathbb{R}^n} \frac{-W_{\tau_0}(y, t_0)}{|x^0 - y|^{n+2s}} \, dy
= -\left( {\partial^\alpha_t} W_{\tau_0}(x^0, t_0) - \Delta W_{\tau_0}(x^0, t_0) \right) \leq 0.
$$
Obviously,
\begin{align}\label{T3305}
W_{\tau_0}(y, t_0) \geq 0, \quad y \in \mathbb{R}^n.
\end{align}
It follows from \eqref{T3305} that
$$
W_{\tau_0}(y, t_0) \equiv 0, \quad y \in \mathbb{R}^n,
$$
since $(x^0, t_0)$ is the maximum point of $W_{\tau_0}(x, t)$ in $\mathbb{R}^n \times \mathbb{R}$, i.e.,
$$
W_{\tau_0}(y, t_0) \leq 0, \quad y \in \mathbb{R}^n.
$$
Consequently, for any $j \in \mathbb{N}$, we have
\begin{align*}
u(y', y_n, t_0)
&= u(y', y_n + \tau_0, t_0)
= u(y', y_n + 2\tau_0, t_0)
= \cdots \\
&= u(y', y_n + j\tau_0, t_0), \quad y \in \mathbb{R}^n.
\end{align*}

By the asymptotic condition, we can choose $y_n$ sufficiently negative so that $u(y', y_n, t_0)$ is close to $-1$, and then take $j$ sufficiently large so that $u(y', y_n + j\tau_0, t_0)$ is close to $1$. This leads to a contradiction and therefore proves \eqref{T3301}, showing that $u(x, t)$ is strictly increasing with respect to $x_n$.

Next, we show that
$$
u(x, t) = u(x_n, t).
$$
If we replace $u_\tau(x, t)$ by $u(x + \tau \mu, t)$ with $\mu = (\mu_1, \cdots, \mu_n)$ and $\mu_n > 0$, the arguments in \textbf{Step 1} and \textbf{Step 2} still hold.

On the one hand, by a similar argument as in \textbf{Step 1} and \textbf{Step 2}, for each $\mu$ with $\mu_n > 0$, we have
$$
u(x + \tau \mu, t) > u(x, t), \quad (x, t) \in \mathbb{R}^{n} \times \mathbb{R}, \quad \forall~\tau > 0.
$$
Letting $\mu_n \to 0$ and using the continuity of $u(x, t)$, it follows that
$$
u(x + \tau \mu, t) \geq u(x, t)
$$
for any $\mu$ with $\mu_n = 0$.

On the other hand, by replacing $\mu$ with $-\mu$, we deduce
$$
u(x + \tau \mu, t) \leq u(x, t)
$$
for any $\mu$ with $\mu_n = 0$.

Hence, we conclude
$$
u(x + \tau \mu, t) = u(x, t),
$$
which implies that $u(x', x_n, t)$ is independent of $x'$. Therefore,
$$
u(x, t) = u(x_n, t).
$$
This completes \textbf{Step 3} and the proof of Theorem \ref{T3}.
\end{proof}

\vspace{2mm}
\noindent \textbf{Acknowledgments.}
Y. Deng was supported by National Natural Science Foundation of China (12271196) and the National Key R\& D Program of China (2023YFA1010002).
P. Wang  was supported by the National Natural Science Foundation of China (No. 12101530), the Sponsored by  Program for Science \& Technology Innovation Talents in Universities of Henan Province (No. 26HASTIT040) and the Nanhu Scholars Program for Young Scholars of XYNU.  Z. Wang was supported by  the Supported by the Scientific Research Foundation of Graduate School of Xinyang Normal University (No. 2025KYJJ50).
L. Wu is partially supported by the National Natural Science Foundation of China (Grant No. 12401133) and the Guangdong Basic and Applied Basic Research Foundation (2025B151502069).

\vspace{2mm}
\noindent \textbf{Conflict of interest.} The authors do not have any possible conflicts of interest.

\vspace{2mm}

\noindent \textbf{Data availability statement.}
 Data sharing is not applicable to this article, as no data sets were generated or analyzed during the current study.

\end{document}